\documentclass[12pt,draft]{article}
\usepackage{amsmath,amscd}
\usepackage{amssymb,latexsym,amsthm}
\usepackage{color}
\usepackage[spanish,english]{babel}
\usepackage{amssymb}
\usepackage{color}
\usepackage{amsmath,amsthm,amscd}
\usepackage[latin1]{inputenc}
\usepackage{lscape}
\usepackage{amsfonts}

\numberwithin{equation}{section}
\newtheorem{theorem}{Theorem}[section]

\newtheorem{proposition}[theorem]{Proposition}
\newtheorem{lemma}[theorem]{Lemma}

\newtheorem{corollary}[theorem]{Corollary}

\theoremstyle{definition}
\newtheorem{example}[theorem]{Example}
\newtheorem{definition}[theorem]{Definition}

\newtheorem{remark}[theorem]{Remark}

\newcommand{\cA}{\mbox{${\cal A}$}}

\newcommand{\cO}{\mbox{${\cal O}$}}

\title{\textbf{Skew PBW extensions over weak symmetric and $(\Sigma,\Delta)$-weak symmetric rings}}
\author{Armando Reyes\\ Department of Mathematics\\ National University of Colombia, Bogot\'a \\ mareyesv@unal.edu.co \\  \\ H\'ector Su\'arez\\ School of Mathematics and Statistics \\ Pedagogical and Technological University of Colombia, Tunja \\ hector.suarez@uptc.edu.co}
\date{}
\begin{document}
\maketitle
\begin{abstract}
\noindent In this paper we study skew Poincar\'e-Birkhoff-Witt extensions over weak symmetric and $(\Sigma,\Delta)$-weak symmetry rings. Since these extensions generalize Ore extensions of injective type and another noncommutative rings of polynomial type, we unify and extend several results in the literature concerning the property of being symmetry. Under adequate conditions, we transfer the property of being weak symmetric or $(\Sigma,\Delta)$-weak symmetric from a ring of coefficients to a skew PBW extension over this ring. We illustrate our results with remarkable examples of algebras appearing in noncommutative algebraic geometry and  theoretical physics.

\bigskip


\bigskip

\bigskip

\end{abstract}

\section{Introduction}\label{section}

A ring $B$ is said to be {\em Armendariz} (the term was introduced by Rege and Chhawchharia \cite{RegeChhawchharia1997}), if whenever polynomials $f(x)=\sum_{i=0}^{s}a_ix^{i}$ and $g(x)=\sum_{j=0}^{t} b_jx^{j}$ in $B[x]$ satisfy $f(x)g(x)=0$, then $a_ib_j$, for all $i, j$. In the context of the well-known Ore extensions, for an endomorphism $\alpha$ and a $\alpha$-derivation $\delta$ of $B$, Moussavi and Hashemi \cite{MoussaviHashemi2005} defined $B$ to be $(\alpha,\delta)$-skew Armendariz, if for $f(x)=\sum_{i=0}^{s}a_ix^{i}$ and $g(x)=\sum_{j=0}^{t} b_jx^{j}$ in $B[x;\alpha,\delta]$ satisfy $f(x)g(x)=0$, then $a_ix^{i}b_jx^{j} = 0$, for each $i, j$. On the other hand, a ring $B$ is called (i) {\em reduced}, if $a^{2}=0\Rightarrow a=0$, for all $a\in B$; (ii) (Lambek \cite{Lambek1971}) {\em symmetric}, if $abc = 0\Rightarrow acb=0$, for all $a,b,c,\in B$; (iii) {\em reversible}, if $ab=0 \Rightarrow ba=0$, for all $a, b\in B$;  (iv) {\em semicommutative}, if $ab=0 \Rightarrow aBb=0$, for all $a, b\in B$ (Bell \cite{Bell1970} defined the following: a ring $B$ is said to satisfy the $IFP$, {\em insertion of factors property}, if $r_B(a)$ is an ideal for all $a\in B$. Sometimes, a semicommutative ring is also called a {a ring with $IFP$ property). It is known that the implications {\em reduced} $\Rightarrow$ {\em symmetric} $\Rightarrow$ {\em reversible} $\Rightarrow$ {\em semicommutative} hold but, in general, the converse of each one of these implications is false, see Marks \cite{Marks2003} for a detailed discussion.\\

Of course, commutative rings are symmetric. Reduced rings are symmetric as we can appreciated in Anderson and Camillo \cite{AndersonCamillo1998}. Nevertheless, there are many nonreduced commutative (so symmetric) rings. Now, if $B$ is Armendariz, then the classical polynomial ring $B[x]$ over $B$ is symmetric if and only if $B$ is symmetric (Huh et al. \cite{HuhLeeSmoktunowicz2002} and Kim and Lee \cite{KimLee2003}. In the noncommutative case, there are results concerning the property of being symmetric over $(\alpha,\delta)$-skew Armendariz rings, see \cite{OuyangChen2010}. Precisely, this was the motivation for Ouyang and Chen who in \cite{OuyangChen2010} defined weak symmetric rings and weak $(\alpha,\delta)$-symmetric rings for the context of Ore extensions $B[x;\alpha,\delta]$, where $B$ is an associate ring with unity. They proved that for every $(\alpha,\delta)$-compatible and reversible ring $B$ (following Annin \cite{Annin2002}, for an endomorphism $\alpha$ and an $\alpha$-derivation $\delta$ of $B$, $B$ is called $\alpha$-compatible, if for every $a,b \in B$, we have $ab=0$ if and only if $a\alpha(b)=0$ (necessarily, the endomorphism $\alpha$ is injective), and $B$ is called to be $\delta$-{\em compatible} if for each $a, b\in B$, $ab=0 \Rightarrow a\delta(b)=0$; if $B$ is both $\alpha$-compatible and $\delta$-compatible, $B$ is called $(\alpha,\delta)$-{\em compatible}), $B$ is weak symmetric if and only if $B[x;\alpha,\delta]$ is weak symmetric, and for every semicommutative ring $B$, $B$ is weak $(\alpha,\delta)$-symmetric if and only if $B[x]$ is weak $(\overline{\alpha}, \overline{\delta})$-symmetric, where $\overline{\alpha}$ and $\overline{\delta}$ are the extended maps of $\alpha$ and $\delta$ over $B[x]$, respectively. The results presented in \cite{OuyangChen2010} generalize those corresponding for .\\

Having in mind all above results and with the aim of establishing similar results to the more general noncommutative rings than (iterated) Ore extensions, in this paper we are interested in the family of rings known as {\em skew Poincar\'e-Birkhoff-Witt extensions} which were introduced by Gallego and Lezama \cite{LezamaGallego2011}. Besides of Ore extensions, skew PBW extensions generalize several families of noncommutative rings (see  \cite{ReyesPhD2013, ReyesYesica2018} for a list of noncommutative rings which are skew PBW extensions but not iterated Ore extensions) and include as particular rings different examples of remarkable algebras appearing in representation theory, Hopf algebras, quantum groups, noncommutative algebraic geometry and another algebras of interest in the context of mathematical physics. Let us mention some of these algebras (see \cite{LezamaReyes2014, ReyesPhD2013} for a detailed reference of every one of these families): (i) universal enveloping algebras of finite dimensional Lie algebras; (ii) PBW extensions introduced by Bell and Goodearl; (iii) almost normalizing extensions defined by McConnell and Robson; (iv) sol\-va\-ble polynomial rings introduced by Kandri-Rody and Weispfenning; (v) diffusion algebras studied by Isaev, Pyatov, and Rittenberg; (vi) 3-dimensional skew polynomial algebras introduced by Bell and Smith; (vii) the regular graded algebras studied by Kirkman, Kuzmanovich, and Zhang, and other noncommutative algebras of polynomial type. The importance of skew PBW extensions is that the coefficients do not necessarily commute with the variables, and these coefficients are not necessarily elements of fields (see Definition \ref {gpbwextension}). In fact, the skew PBW extensions contain well-known groups of algebras such as some types of $G$-algebras introduced by Apel and some PBW algebras defined by Bueso et. al., (both $G$-algebras and PBW algebras take coefficients in fields and assume that coefficientes commute with variables), Auslander-Gorenstein rings, some Calabi-Yau and skew Calabi-Yau algebras, some Artin-Schelter regular algebras, some Koszul and augmented Koszul algebras, quantum polynomials, some quantum universal enveloping algebras, some graded skew Clifford algebras and others (see \cite{BrownGoodearl2002, BuesoTorrecillasVerschoren, Rosenberg1995, Suarez2017, SuarezLezamaReyes2017, SuarezReyesgenerKoszul2017} for a list of examples). As we can appreciated, skew PBW extensions include a lot of noncommutative rings, which means that a theory of symmetry for these extensions will cover several treatments in the literature and will establish similar results for algebras not considered before. To formulate this theory is the objective of the present paper. In this way, we continue the study of ring theoretical properties of skew PBW extensions (c.f. \cite{AcostaLezama2015, Artamonov2015, ArtamonovLezamaFajardo2016, LezamaGallego2011, LezamaGallego2017, LezamaAcostaReyes2015,  LezamaLatorre2017, LezamaReyes2014, ReyesPhD2013}).\\

The paper is organized as follows: In Section \ref{definitionexamplesspbw} we establish some useful results about skew PBW extensions for the rest of the paper. In Section \ref{SigmarigidandSigmaDeltacompatible} we recall the notions of $\Sigma$-rigid rings and $(\Sigma,\Delta)$-compatible rings which are key throughout the paper. Next, in Section  \ref{weaksymmetricskewPBWextensions} we present some results about nilpotent elements of skew PBW extensions and then characterize these extensions over weak symmetric rings. In Section \ref{SigmaDeltaweaksymmetricskewPBWextensions}  we investigate skew PBW extensions over weak $(\Sigma,\Delta)$-symmetric rings. The results presented in Sections \ref{weaksymmetricskewPBWextensions} and \ref{SigmaDeltaweaksymmetricskewPBWextensions} generalize corresponding results presented by Ouyang and Chen \cite{OuyangChen2010} for Ore extensions of injective type and generalize those presented in \cite{JaramilloReyes2018}. The techniques used here are fairly standard and follow the same path as other text on the subject. Finally, Section \ref{examplespaper} presents remarkable examples appearing in noncommutative algebraic geometry and theoretical physics where results obtained in Sections \ref{weaksymmetricskewPBWextensions} and \ref{SigmaDeltaweaksymmetricskewPBWextensions} can be illustrated. \\

Throughout the paper, the word ring means a ring (not necessarily commutative) with unity. The letter $k$ will denote a commutative ring and $\Bbbk$ will denote a field. $\mathbb{C}$ will denote the field of complex numbers. Finally, for a ring $B$, ${\rm nil}(B)$ represents the set of nilpotent elements of $B$.
\section{Skew PBW extensions}\label{definitionexamplesspbw}
In this section we recall some  results about skew PBW extensions which are important for the rest of the paper.
\begin{definition}[\cite{LezamaGallego2011},  Definition 1]\label{gpbwextension}
Let $R$ and $A$ be rings. We say that $A$ is a {\em skew PBW extension} (also known as {\em $\sigma$-PBW extension}) {\em of}  $R$, which is denoted by $A:=\sigma(R)\langle
x_1,\dots,x_n\rangle$, if the following conditions hold:
\begin{enumerate}
\item[\rm (i)]$R\subseteq A$.
\item[\rm (ii)]there exist elements $x_1,\dots ,x_n\in A$ such that $A$ is a left free $R$-module, with basis ${\rm Mon}(A):= \{x^{\alpha}=x_1^{\alpha_1}\cdots
x_n^{\alpha_n}\mid \alpha=(\alpha_1,\dots ,\alpha_n)\in
\mathbb{N}^n\}$,  and $x_1^{0}\dotsb x_n^{0}:=1\in {\rm Mon}(A)$.

\item[\rm (iii)]For each $1\leq i\leq n$ and any $r\in R\ \backslash\ \{0\}$, there exists an element $c_{i,r}\in R\ \backslash\ \{0\}$ such that $x_ir-c_{i,r}x_i\in R$.
\item[\rm (iv)] For any elements $1\leq i,j\leq n$, there exists $c_{i,j}\in R\ \backslash\ \{0\}$ such that $x_jx_i-c_{i,j}x_ix_j\in R+Rx_1+\cdots +Rx_n$ (i.e., there exist elements $r_0^{(i,j)}, r_1^{(i,j)}, \dotsc, r_n^{(i,j)}$ of $R$ with $x_jx_i - c_{i,j}x_ix_j = r_0^{(i,j)} + \sum_{k=1}^{n} r_k^{(i,j)}x_k$).
\end{enumerate}
\end{definition}
Since ${\rm Mon}(A)$ is a left $R$-basis of $A$, the elements $c_{i,r}$ and $c_{i,j}$ are unique, (\cite{LezamaGallego2011}, Remark 2).
\begin{proposition}[\cite{LezamaGallego2011}, Proposition 3]\label{sigmadefinition}
Let $A$ be a skew PBW  extension of $R$. For each $1\leq i\leq
n$, there exist an injective endomorphism $\sigma_i:R\rightarrow
R$ and an $\sigma_i$-derivation $\delta_i:R\rightarrow R$ such that $x_ir=\sigma_i(r)x_i+\delta_i(r)$, for  each $r\in R$. From now on, we will write  $\Sigma:=\{\sigma_1,\dotsc, \sigma_n\}$, and $\Delta:=\{\delta_1,\dotsc, \delta_n\}$.
\end{proposition}
\begin{definition}[\cite{LezamaGallego2011}, Definition 4; \cite{LezamaAcostaReyes2015}, Definition 2.3]\label{sigmapbwderivationtype}
Let $A$ be a skew PBW extension of $R$.
\begin{enumerate}
\item[\rm (a)] $A$ is called \textit{quasi-commutative}, if the conditions {\rm(}iii{\rm)} and {\rm(}iv{\rm)} in Definition \ref{gpbwextension} are replaced by the following: (iii') for each $1\leq i\leq n$ and all $r\in R\ \backslash\ \{0\}$, there exists $c_{i,r}\in R\ \backslash\ \{0\}$ such that $x_ir=c_{i,r}x_i$; (iv') for any $1\leq i,j\leq n$, there exists $c_{i,j}\in R\ \backslash\ \{0\}$ such that $x_jx_i=c_{i,j}x_ix_j$.
\item[\rm (b)] $A$ is called \textit{bijective}, if $\sigma_i$ is bijective for each $1\leq i\leq n$, and $c_{i,j}$ is invertible, for any $1\leq i<j\leq n$.
\item[\rm (c)] $A$ is called of {\em endomorphism type}, if $\delta_i=0$, for every $i$.  In addition, if every $\sigma_i$ is bijective, $A$ is a skew PBW extension of {\em automorphism type}.
\end{enumerate}
\end{definition}
\begin{example}\label{mentioned}
If $R[x_1;\sigma_1,\delta_1]\dotsb [x_n;\sigma_n,\delta_n]$ is an iterated Ore extension where
\begin{itemize}
\item $\sigma_i$ is injective, for $1\le i\le n$,
\item $\sigma_i(r)$, $\delta_i(r)\in R$, for every $r\in R$ and $1\le i\le n$,
\item $\sigma_j(x_i)=cx_i+d$, for $i < j$, and $c, d\in R$, where $c$ has a left inverse,
\item $\delta_j(x_i)\in R + Rx_1 + \dotsb + Rx_n$, for $i < j$,
\end{itemize}
then $R[x_1;\sigma_1,\delta_1]\dotsb [x_n;\sigma_n, \delta_n] \cong \sigma(R)\langle x_1,\dotsc, x_n\rangle$ (\cite{LezamaReyes2014}, p. 1212). Skew PBW extensions of endomorphism type are more general than iterated Ore extensions of the form $R[x_1;\sigma_1]\dotsb [x_n;\sigma_n]$, and in general, skew PBW extensions are more general than Ore extensions of injective type (see \cite{LezamaReyes2014}). Examples of noncommutative rings which are skew PBW extensions but can not be expressed as iterated Ore extensions can be found in \cite{ReyesPhD2013, ReyesYesica2018}.
\end{example}
\begin{definition}\label{definitioncoefficients}
If $A$ is a skew PBW extension of $R$, then:
\begin{enumerate}
\item[\rm (i)]for $\alpha=(\alpha_1,\dots,\alpha_n)\in \mathbb{N}^n$,
$\sigma^{\alpha}:=\sigma_1^{\alpha_1}\cdots \sigma_n^{\alpha_n}$,
$|\alpha|:=\alpha_1+\cdots+\alpha_n$. If
$\beta=(\beta_1,\dots,\beta_n)\in \mathbb{N}^n$, then
$\alpha+\beta:=(\alpha_1+\beta_1,\dots,\alpha_n+\beta_n)$.
\item[\rm (ii)]For $X=x^{\alpha}\in {\rm Mon}(A)$,
$\exp(X):=\alpha$, $\deg(X):=|\alpha|$, and $X_0:=1$. The symbol $\succeq$ will denote a total order defined on ${\rm Mon}(A)$ (a total order on $\mathbb{N}^n$). For an
 element $x^{\alpha}\in {\rm Mon}(A)$, ${\rm exp}(x^{\alpha}):=\alpha\in \mathbb{N}^n$.  If
$x^{\alpha}\succeq x^{\beta}$ but $x^{\alpha}\neq x^{\beta}$, we
write $x^{\alpha}\succ x^{\beta}$. Every element $f\in A$ can be expressed uniquely as $f=a_0 + a_1X_1+\dotsb +a_mX_m$, with $a_i\in R$, and $X_m\succ \dotsb \succ X_1$ (eventually, we will use expressions as $f=a_0 + a_1Y_1+\dotsb +a_mY_m$, with $a_i\in R$, and $Y_m\succ \dotsb \succ Y_1$). With this notation, we define ${\rm
lm}(f):=X_m$, the \textit{leading monomial} of $f$; ${\rm
lc}(f):=a_m$, the \textit{leading coefficient} of $f$; ${\rm
lt}(f):=a_mX_m$, the \textit{leading term} of $f$; ${\rm exp}(f):={\rm exp}(X_m)$, the \textit{order} of $f$; and
 $E(f):=\{{\rm exp}(X_i)\mid 1\le i\le t\}$. Note that $\deg(f):={\rm max}\{\deg(X_i)\}_{i=1}^t$. Finally, if $f=0$, then
${\rm lm}(0):=0$, ${\rm lc}(0):=0$, ${\rm lt}(0):=0$. We also
consider $X\succ 0$ for any $X\in {\rm Mon}(A)$. For a detailed description of monomial orders in skew PBW  extensions, see \cite{LezamaGallego2011}, Section 3.
\end{enumerate}
\end{definition}
\begin{proposition}[\cite{LezamaGallego2011}, Theorem 7]\label{coefficientes}
If $A$ is a polynomial ring with coefficients in $R$ with respect to the set of indeterminates $\{x_1,\dots,x_n\}$, then $A$ is a skew PBW  extension of $R$ if and only if the following conditions hold:
\begin{enumerate}
\item[\rm (1)]for each $x^{\alpha}\in {\rm Mon}(A)$ and every $0\neq r\in R$, there exist unique elements $r_{\alpha}:=\sigma^{\alpha}(r)\in R\ \backslash\ \{0\}$, $p_{\alpha ,r}\in A$, such that $x^{\alpha}r=r_{\alpha}x^{\alpha}+p_{\alpha, r}$,  where $p_{\alpha ,r}=0$, or $\deg(p_{\alpha ,r})<|\alpha|$ if
$p_{\alpha , r}\neq 0$. If $r$ is left invertible,  so is $r_\alpha$.
\item[\rm (2)]For each $x^{\alpha},x^{\beta}\in {\rm Mon}(A)$,  there exist unique elements $c_{\alpha,\beta}\in R$ and $p_{\alpha,\beta}\in A$ such that $x^{\alpha}x^{\beta}=c_{\alpha,\beta}x^{\alpha+\beta}+p_{\alpha,\beta}$, where $c_{\alpha,\beta}$ is left invertible, $p_{\alpha,\beta}=0$, or $\deg(p_{\alpha,\beta})<|\alpha+\beta|$ if
$p_{\alpha,\beta}\neq 0$.
\end{enumerate}
\end{proposition}
\begin{proposition}[\cite{Reyes2015},  Proposition 2.9] \label{lindass}
If $\alpha=(\alpha_1,\dotsc, \alpha_n)\in \mathbb{N}^{n}$ and $r$ is an element of $R$, then
{\small{\begin{align*}
x^{\alpha}r = &\ x_1^{\alpha_1}x_2^{\alpha_2}\dotsb x_{n-1}^{\alpha_{n-1}}x_n^{\alpha_n}r = x_1^{\alpha_1}\dotsb x_{n-1}^{\alpha_{n-1}}\biggl(\sum_{j=1}^{\alpha_n}x_n^{\alpha_{n}-j}\delta_n(\sigma_n^{j-1}(r))x_n^{j-1}\biggr)\\
+ &\ x_1^{\alpha_1}\dotsb x_{n-2}^{\alpha_{n-2}}\biggl(\sum_{j=1}^{\alpha_{n-1}}x_{n-1}^{\alpha_{n-1}-j}\delta_{n-1}(\sigma_{n-1}^{j-1}(\sigma_n^{\alpha_n}(r)))x_{n-1}^{j-1}\biggr)x_n^{\alpha_n}\\
+ &\ x_1^{\alpha_1}\dotsb x_{n-3}^{\alpha_{n-3}}\biggl(\sum_{j=1}^{\alpha_{n-2}} x_{n-2}^{\alpha_{n-2}-j}\delta_{n-2}(\sigma_{n-2}^{j-1}(\sigma_{n-1}^{\alpha_{n-1}}(\sigma_n^{\alpha_n}(r))))x_{n-2}^{j-1}\biggr)x_{n-1}^{\alpha_{n-1}}x_n^{\alpha_n}\\
+ &\ \dotsb + x_1^{\alpha_1}\biggl( \sum_{j=1}^{\alpha_2}x_2^{\alpha_2-j}\delta_2(\sigma_2^{j-1}(\sigma_3^{\alpha_3}(\sigma_4^{\alpha_4}(\dotsb (\sigma_n^{\alpha_n}(r))))))x_2^{j-1}\biggr)x_3^{\alpha_3}x_4^{\alpha_4}\dotsb x_{n-1}^{\alpha_{n-1}}x_n^{\alpha_n} \\
+ &\ \sigma_1^{\alpha_1}(\sigma_2^{\alpha_2}(\dotsb (\sigma_n^{\alpha_n}(r))))x_1^{\alpha_1}\dotsb x_n^{\alpha_n}, \ \ \ \ \ \ \ \ \ \ \sigma_j^{0}:={\rm id}_R\ \ {\rm for}\ \ 1\le j\le n.
\end{align*}}}
\end{proposition}
\begin{remark}[\cite{Reyes2015}, Remark 2.10 (iv)] \label{juradpr}
About Proposition \ref{lindass}, we have the following observation: if $X_i:=x_1^{\alpha_{i1}}\dotsb x_n^{\alpha_{in}}$ and $Y_j:=x_1^{\beta_{j1}}\dotsb x_n^{\beta_{jn}}$, when we compute every summand of $a_iX_ib_jY_j$ we obtain pro\-ducts of the coefficient $a_i$ with several evaluations of $b_j$ in $\sigma$'s and $\delta$'s depending of the coordinates of $\alpha_i$. This assertion follows from the expression:
\begin{align*}
a_iX_ib_jY_j = &\ a_i\sigma^{\alpha_{i}}(b_j)x^{\alpha_i}x^{\beta_j} + a_ip_{\alpha_{i1}, \sigma_{i2}^{\alpha_{i2}}(\dotsb (\sigma_{in}^{\alpha_{in}}(b_j)))} x_2^{\alpha_{i2}}\dotsb x_n^{\alpha_{in}}x^{\beta_j} \\
+ &\ a_i x_1^{\alpha_{i1}}p_{\alpha_{i2}, \sigma_3^{\alpha_{i3}}(\dotsb (\sigma_{{in}}^{\alpha_{in}}(b_j)))} x_3^{\alpha_{i3}}\dotsb x_n^{\alpha_{in}}x^{\beta_j} \\
+ &\ a_i x_1^{\alpha_{i1}}x_2^{\alpha_{i2}}p_{\alpha_{i3}, \sigma_{i4}^{\alpha_{i4}} (\dotsb (\sigma_{in}^{\alpha_{in}}(b_j)))} x_4^{\alpha_{i4}}\dotsb x_n^{\alpha_{in}}x^{\beta_j}\\
+ &\ \dotsb + a_i x_1^{\alpha_{i1}}x_2^{\alpha_{i2}} \dotsb x_{i(n-2)}^{\alpha_{i(n-2)}}p_{\alpha_{i(n-1)}, \sigma_{in}^{\alpha_{in}}(b_j)}x_n^{\alpha_{in}}x^{\beta_j} \\
+ &\ a_i x_1^{\alpha_{i1}}\dotsb x_{i(n-1)}^{\alpha_{i(n-1)}}p_{\alpha_{in}, b_j}x^{\beta_j}.
\end{align*}
\end{remark}
\section{$\Sigma$-rigid rings and $(\Sigma,\Delta)$-compatible rings}\label{SigmarigidandSigmaDeltacompatible}
In this section we recall some results concerning $\Sigma$-rigid rings and $(\Sigma,\Delta)$-compatible rings and their relation with skew PBW extensions.
\begin{definition}[\cite{Reyes2015},  Definition 3.2] \label{generaldef2015}
Let $B$ be a ring and $\Sigma$ a family of endomorphisms of $B$. $\Sigma$ is called a {\em rigid endomorphisms family}, if $r\sigma^{\alpha}(r)=0$ implies $r=0$,  for every $r\in B$ and $\alpha\in \mathbb{N}^n$. A ring $B$ is called to be $\Sigma$-{\em rigid}, if there exists a rigid endomorphisms family $\Sigma$ of $B$.
\end{definition}
The motivation to define $\Sigma$-rigid rings was to generalize the rigid rings defined by Krempa \cite{Krempa1996}. Now, if $\Sigma$ is a rigid endomorphisms family, then every element $\sigma_i\in \Sigma$ is a monomorphism. In fact, $\Sigma$-rigid rings are reduced rings: if $B$ is a $\Sigma$-rigid ring and $r^2=0$ for $r\in B$, then we have the equalities $0=r\sigma^{\alpha}(r^2)\sigma^{\alpha}(\sigma^{\alpha}(r))=r\sigma^{\alpha}(r)\sigma^{\alpha}(r)\sigma^{\alpha}(\sigma^{\alpha}(r))=r\sigma^{\alpha}(r)\sigma^{\alpha}(r\sigma^{\alpha}(r))$, i.e., $r\sigma^{\alpha}(r)=0$ and so $r=0$, that is, $B$ is reduced (note that there exists an endomorphism of a reduced ring which is not a rigid endomorphism, see \cite{HongKimKwak2000}, Example 9). With this in mind, we consider the family of injective endomorphisms $\Sigma$ and the family $\Delta$ of $\Sigma$-derivations in a skew PBW extension $A$ of a ring $R$ (Proposition \ref{sigmadefinition}). The notion of rigidness with another ring theoretical properties such as Baer, quasi-Baer, p.p and p.q have been investigated for skew PBW extensions in \cite{NinoReyes2017, Reyes2015, Reyes2018, ReyesSuarezClifford2017, ReyesSuarezUMA2018, ReyesYesica2018} (in the context of Ore extensions, the beautiful monograph \cite{Birkenmeieretal2013} contains a complete list of works on all these properties). Recall that if $A$ is a skew PBW extension of $R$ where the the elements $c_{i, j}$ are in\-ver\-ti\-ble in $R$, then $R$ is $\Sigma$-rigid if and only if $A$ is a reduced ring (\cite{Reyes2015},  Proposition 3.5).
\begin{proposition}[\cite{Reyes2015}, Lemma 3.3 and Corollary 3.4]\label{Reyes2015Lemma3.3}
Let $R$ be an $\Sigma$-rigid ring and $a,b\in R$. Then:
\begin{enumerate}
\item [\rm (1)] If $ab=0$ then $a\sigma^{\alpha}(b)=\sigma^{\alpha}(a)b=0$, for any $\alpha\in \mathbb{N}^n$.
\item [\rm (2)] If $ab=0$ then $a\delta^{\beta}(b)=\delta^{\beta}(a)b=0$, for any $\beta\in \mathbb{N}^n$.
\item [\rm (3)] If $ab=0$ then $a\sigma^{\alpha}(\delta^{\beta}(b))=a\delta^{\beta}(\sigma^{\alpha}(b))=0$,  for every $\alpha, \beta\in \mathbb{N}^n$.
\item [\rm (4)] If $a\sigma^{\theta}(b)=\sigma^{\theta}(a)b=0$ for some $\theta\in \mathbb{N}^n$, then $ab=0$.
\item [\rm (5)] If $A$ is a skew PBW  extension over $R$, $ab=0\Rightarrow  ax^{\alpha}bx^{\beta}=0$, for any elements $a,b\in R$ and each $\alpha, \beta\in \mathbb{N}^n$.
\end{enumerate}
\end{proposition}
Next we present the notion of $(\Sigma,\Delta)$-compatible rings which was introduced by the authors in \cite{ReyesSuarezUMA2018}.
\begin{definition}[\cite{ReyesSuarezUMA2018}, Definition 3.2]\label{Definition3.52008}
Consider a ring $R$ with a family of endomorphisms  $\Sigma$ and a family of $\Sigma$-derivations $\Delta$. Then,
\begin{enumerate}
\item [\rm (i)] $R$ is said to be $\Sigma$-{\em compatible}, if for each $a,b\in R$, $a\sigma^{\alpha}(b)=0$ if and only if $ab=0$, for every $\alpha\in \mathbb{N}^{n}$;
\item [\rm (ii)] $R$ is said to be $\Delta$-{\em compatible}, if for each $a,b \in R$,  $ab=0$ implies $a\delta^{\beta}(b)=0$, for every $\beta \in \mathbb{N}^{n}$.
\end{enumerate}
If $R$ is both $\Sigma$-compatible and $\Delta$-compatible, $R$ is called $(\Sigma, \Delta)$-{\em compatible}.
\end{definition}
\begin{proposition}[\cite{ReyesSuarezUMA2018}, Proposition 3.8]\label{ReyesSuarezUMA2017Prop3.8}
Let $R$ be a $(\Sigma, \Delta)$-compatible ring. For every $a, b \in R$, we have:
\begin{enumerate}
\item [\rm (1)] if $ab=0$, then $a\sigma^{\theta}(b) = \sigma^{\theta}(a)b=0$, for each $\theta\in \mathbb{N}^{n}$.
\item [\rm (2)] If $\sigma^{\beta}(a)b=0$ for some $\beta\in \mathbb{N}^{n}$, then $ab=0$.
\item [\rm (3)] If $ab=0$, then $\sigma^{\theta}(a)\delta^{\beta}(b)= \delta^{\beta}(a)\sigma^{\theta}(b) = 0$, for every $\theta, \beta\in \mathbb{N}^{n}$.
\end{enumerate}
\end{proposition}
From \cite{ReyesSuarezUMA2018}, Proposition 3.4, we know that every $\Sigma$-rigid ring is a $(\Sigma, \Delta)$-compatible ring. The converse is false as we can appreciated in \cite{ReyesSuarezUMA2018}, Example 3.6. In this way, $\Sigma$-rigid rings are contained strictly in $(\Sigma, \Delta)$-compatible rings. Nevertheless, these two notions coincide when the ring is assumed to be reduced, such as the following proposition establishes.
\begin{proposition}[\cite{ReyesSuarezUMA2018}, Theorem 3.9]\label{ReyesSuarez2018Theorem3.9}
If $A$ is a skew PBW extension of a ring $R$, then the following statements are equivalent: (1) $R$ is reduced and $(\Sigma, \Delta)$-compatible. (2) $R$ is $\Sigma$-rigid. (3) $A$ is reduced.
\end{proposition}
\section{Skew PBW extensions over weak symmetric rings}\label{weaksymmetricskewPBWextensions}
In \cite{OuyangChen2010}, Definition 1, Ouyang and Chen introduced the notion of weak symmetric ring in the following way: a ring $B$ is called a {\em weak symmetric} ring, if $abc\in {\rm nil}(B) \Rightarrow acb\in {\rm nil}(B)$, for every elements $a, b, c\in R$. They proved that their notion extends the concept of symmetric ring, that is all symmetric rings are weak symmetric (\cite{OuyangChen2010}, Proposition 2.1). However, the converse of the assertion is false, i.e, there exists a weak symmetric ring which is not symmetric (\cite{OuyangChen2010}, Example 2.2). \\

With the aim of studying these notions of symmetry in the case of skew PBW extensions, we start with four results (Lemmas \ref{2010Lemma2.7} and \ref{2010Lemma2.8} and Theorems \ref{2010Lemma2.10} and \ref{2010Theorem2.11})  about nilpotent elements in skew PBW extensions. Our Lemma \ref{2010Lemma2.7} generalizes \cite{OuyangChen2010}, Lemma 2.7.
\begin{lemma}\label{2010Lemma2.7}
If $R$ is a $(\Sigma,\Delta)$-compatible and reversible ring, then $ab\in {\rm nil}(R)$ implies that $a\sigma^{\alpha}(\delta^{\beta}(b))$ and $a\delta^{\beta}(\sigma^{\alpha}(b))$ also are elements of ${\rm nil}(R)$, for any $\alpha,\beta\in \mathbb{N}^{n}$.
\end{lemma}
\begin{proof}
By assumption there exists a positive integer $k$ such that $(ab)^{k}=0$. Consider the following equalities:
\begin{align*}
(ab)^{k} = &\ ab ab \dotsb ab ab ab\ \ \ (k\ {\rm times})\\
= &\ abab \dotsb ab ab a\sigma^{\alpha}(\delta^{\beta}(b))\ \ \ ({\rm Proposition\ \ref{ReyesSuarezUMA2017Prop3.8}\ (3)})\\
= &\ a\sigma^{\alpha}(\delta^{\beta}(b)) ab ab ab \dotsb ab ab\ \ \ (R\ {\rm is\ reversible})\\
= &\ a\sigma^{\alpha}(\delta^{\beta}(b)) ab ab \dotsb ab a\sigma^{\alpha}(\delta^{\beta}(b))\ \ \ ({\rm Proposition\ \ref{ReyesSuarezUMA2017Prop3.8}\ (3)})\\
= &\ a\sigma^{\alpha}(\delta^{\beta}(b)) a\sigma^{\alpha}(\delta^{\beta}(b)) ab ab \dotsb ab\ \ \ (R\ {\rm is\ reversible})
\end{align*}
Following this procedure we guarantee that the element $a\sigma^{\alpha}(\delta^{\beta}(b))$ belongs to ${\rm nil}(R)$. For the element $a\delta^{\beta}(\sigma^{\alpha}(b))$ the reasoning is completely similar. 
\end{proof}
The next lemma extends \cite{OuyangChen2010}, Lemma 2.8.
\begin{lemma}\label{2010Lemma2.8}
If $R$ is a $(\Sigma,\Delta)$-compatible ring, then $a\sigma^{\theta}(b)\in {\rm nil}(R)$ implies $ab\in {\rm nil}(R)$, for every $\theta \in \mathbb{N}^{n}$ and each $a, b \in R$.
\end{lemma}
\begin{proof}
Since $a\sigma^{\theta}(b)\in {\rm nil}(R)$, there exists a positive integer $k$ with $(a\sigma^{\theta}(b))^{k}=0$. We have the following assertions:
\begin{align*}
(a\sigma^{\theta}(b))^{k} =  &\ a\sigma^{\theta}(b) a\sigma^{\theta}(b) \dotsb a\sigma^{\theta}(b) a\sigma^{\theta}(b)\ \ \ (k\ {\rm times})\\
= &\ a\sigma^{\theta}(b)a\sigma^{\theta}(b) \dotsb a\sigma^{\theta}(b) ab \ \ \ ({\rm Definition\ of}\ \Sigma-{\rm compatibility})\\
= &\ a\sigma^{\theta}(b) a\sigma^{\theta}(b) \dotsb a\sigma^{\theta}(b) \sigma^{\theta}(ab)\ \ \ ({\rm Proposition\ \ref{ReyesSuarezUMA2017Prop3.8}\ (1)})\\
= &\ a\sigma^{\theta}(b) a\sigma^{\theta}(b) \dotsb a \sigma^{\theta}(bab)\ \ \ (\sigma^{\theta}\ {\rm is\ an\ endomorphism\ of}\ R)\\
= &\ a\sigma^{\theta}(b) a\sigma^{\theta}(b) \dotsb abab\ \ \
({\rm Definition\ of}\ \Sigma-{\rm compatibility})
\end{align*}
If we continue in this way, we can see that the element $ab\in {\rm nil}(R)$, which concludes the proof.
\end{proof}
We recall from \cite{LiuZhao2006}, Lemma 3.1, that if $B$ is a semicommutative ring, then ${\rm nil}(B)$ is an ideal of $B$. Our Theorem \ref{2010Lemma2.10} generalizes \cite{OuyangChen2010}, Lemma 2.10. We need to assume that the elements $c_{i,j}$ of Definition \ref{gpbwextension} (iv) are central in $R$. With the purpose of abbreviating, we will write o.t.l.t to mean {\em other terms less than} in the sense of monomial orders (Definition \ref{definitioncoefficients} (ii)).
\begin{theorem}\label{2010Lemma2.10} If $A$ is a skew PBW extension over a $(\Sigma, \Delta)$-compatible and reversible ring $R$, then for every element $f=\sum_{i=0}^{m} a_iX_i\in A$, $f\in {\rm nil}(A)$ if and only if $a_i\in {\rm nil}(R)$, for each $1\le i\le m$.
\end{theorem}
\begin{proof}
Let $f\in A$ given as above and suppose that $f\in {\rm nil}(A)$ with $X_1 \prec X_2 \prec \dotsb \prec X_m$. Consider the notation established in Proposition \ref{coefficientes}. There exists a positive integer $k$ such that $f^{k} = (a_0 + a_1X_1 + \dotsb + a_mX_m)^{k}=0$. As an illustration, note that
{\normalsize{\begin{align*}
f^{2} = &\ (a_mX_m + \dotsb + a_1X_1 + a_0)(a_mX_m + \dotsb + a_1X_1 + a_0)\\
= &\ a_mX_ma_mX_m +\ {\rm o.t.l.t}\ {\rm exp}(X_m)\\
= &\ a_m[\sigma^{\alpha_m}(a_m)X_m + p_{\alpha_m, a_m}]X_m + \ {\rm o.t.l.t}\ {\rm exp}(X_m)\\
= &\ a_m\sigma^{\alpha_m}(a_m)X_mX_m + a_mp_{\alpha_m, a_m}X_m +\ {\rm o.t.l.t}\ {\rm exp}(X_m)\\
= &\ a_m\sigma^{\alpha_m}(a_m)[c_{\alpha_m, \alpha_m}x^{2\alpha_m} + p_{\alpha_m, \alpha_m}] + a_mp_{\alpha_m, a_m}X_m + {\rm o.t.l.t}\ {\rm exp}(X_m)\\
= &\ a_m\sigma^{\alpha_m}(a_m)c_{\alpha_m, \alpha_m}x^{2\alpha_m} + \ {\rm o.t.l.t}\ {\rm exp}(x^{2\alpha_m}),
\end{align*}}}
and hence,
{\normalsize{\begin{align*}
f^{3} = &\ (a_m\sigma^{\alpha_m}(a_m)c_{\alpha_m, \alpha_m}x^{2\alpha_m} + \ {\rm o.t.l.t}\ {\rm exp}(x^{2\alpha_m})) (a_mX_m + \dotsb + a_1x_1 + a_0)\\
= &\ a_m\sigma^{\alpha_m}(a_m)c_{\alpha_m, \alpha_m}x^{2\alpha_m}a_mX_m + \ {\rm o.t.l.t}\ {\rm exp}(x^{3\alpha_m})\\
= &\ a_m\sigma^{\alpha_m}(a_m)c_{\alpha_m, \alpha_m}[\sigma^{2\alpha_m}(a_m)x^{2\alpha_m} + p_{2\alpha_m, a_m}]X_m + \ {\rm o.t.l.t}\ {\rm exp}(x^{3\alpha_m})\\
= &\ a_m\sigma^{\alpha_m}(a_m)c_{\alpha_m,\alpha_m}\sigma^{2\alpha_m}(a_m)x^{2\alpha_m}X_m + \ {\rm o.t.l.t}\ {\rm exp}(x^{3\alpha_m})\\
= &\ a_m\sigma^{\alpha_m}(a_m)c_{\alpha_m,\alpha_m}\sigma^{2\alpha_m}(a_m)[c_{2\alpha_m, \alpha_m}x^{3\alpha_m} + p_{2\alpha_m,\alpha_m}]\\
= &\ a_m\sigma^{\alpha_m}(a_m)c_{\alpha_m,\alpha_m}\sigma^{2\alpha_m}(a_m)c_{2\alpha_m, \alpha_m}x^{3\alpha_m} + \ {\rm o.t.l.t}\ {\rm exp}(x^{3\alpha_m}).
\end{align*}}}
Continuing in this way, one can show that for $f^{k}$,
\[
f^{k} = \biggl \{a_m\prod_{l=1}^{k-1}\sigma^{l\alpha_m}(a_m)c_{l\alpha_m, \alpha_m}x^{k\alpha_m}\biggr\} + \ {\rm o.t.l.t}\ {\rm exp}(x^{k\alpha_m}),
\]
whence $0={\rm lc}(f^{k}) = a_m\prod_{l=1}^{k-1}\sigma^{l\alpha_m}(a_m)c_{l\alpha_m, \alpha_m}$, and since the elements $c$'s are central in $R$ and left invertible (Proposition \ref{coefficientes}), we have $0={\rm lc}(f^{k}) = a_m\prod_{l=1}^{k-1}\sigma^{l\alpha_m}(a_m)$. Using the $\Sigma$-compatibility of $R$, we obtain $a_m\in {\rm nil}(R)$.

Now, since
{\small{\begin{align*}
f^{k} = &\ ((a_0 + a_1X_1 + \dotsb + a_{m-1}X_{m-1}) + a_mX_m)^{k} \\
= &\ ((a_0 + a_1X_1 + \dotsb + a_{m-1}X_{m-1}) + a_mX_m)((a_0 + a_1X_1 + \dotsb + a_{m-1}X_{m-1}) + a_mX_m)\\
&\ \dotsb ((a_0 + a_1X_1 + \dotsb + a_{m-1}X_{m-1}) + a_mX_m)\ \ \ \ \ \ (k\ {\rm times})\\
= &\ [(a_0 + a_1X_1 + \dotsb + a_{m-1}X_{m-1})^{2} + (a_0 + a_1X_1 + \dotsb + a_{m-1}X_{m-1})a_mX_m\\
&\ + a_mX_m(a_0 + a_1X_1 + \dotsb + a_{m-1}X_{m-1}) + a_mX_ma_mX_m]\\
&\ \dotsb ((a_0 + a_1X_1 + \dotsb + a_{m-1}X_{m-1}) + a_mX_m)\\
= &\ (a_0 + a_1X_1 + \dotsb + a_{m-1}X_{m-1})^{k} + h,
\end{align*}}}
where $h$ is an element of $A$ which involves products of monomials with the term $a_mX_m$ on the left and the right, by  Proposition \ref{lindass}, Remark \ref{juradpr} and having in mind that $a_m\in {\rm nil}(R)$, which is an ideal of $R$ (remember that reversible implies semicommutative), the expression for $f^{k}$ reduces to $f^k = (a_0 + a_1X_1 + \dotsb + a_{m-1}X_{m-1})^{k}$. Using a similar reasoning as above, one can prove that
\[
f^{k} = a_{m-1}\prod_{l=1}^{k-1}\sigma^{l(\alpha_{m-1})}(a_{m-1})c_{l(\alpha_{m-1}), \alpha_{m-1}}x^{k\alpha_{m-1}} + \ {\rm o.t.l.t}\ {\rm exp}(x^{k\alpha_{m-1}}).
\]
Hence ${\rm lc}(f^{k}) = a_{m-1}\prod_{l=1}^{k-1}\sigma^{l\alpha_{m-1}}(a_{m-1})c_{l\alpha_{m-1}, \alpha_{m-1}}$, and so $a_{m-1}\in {\rm nil}(R)$. If we repeat this argument, it follows that $a_i\in {\rm nil}(R)$, for $0\le i\le m$.

Conversely, suppose that $a_i\in {\rm nil}(R)$, for every $i$. If $k_i$ is the minimum integer positive such that $a_i^{k_i} = 0$, for every $i$, let $k:={\rm max}\{k_i\mid 1\le i\le n\}$. It is clear that $a_i^{k}=0$, for all $i$. Let us prove that $f^{(m+1){k}+1} = 0$, and hence, $f\in {\rm nil}(A)$. Since the  expression for $f$ have $m+1$ terms, when we realize the product $f^{(m+1){k}+1}$ we have sums of products of the form
\begin{equation}\label{rigoo}
a_{i,1}X_{i,1}a_{i,2}X_{i,2}\dotsb a_{i, (m+1){k}}X_{i, (m+1){k}}a_{i,(m+1){k}+1}X_{i,(m+1){k}+1}.
\end{equation}
Note that there are exactly $(m+1)^{(m+1)k+1}$ products of the form (\ref{rigoo}). Now, since when we compute $f^{(m+1){k}+1}$ every product as (\ref{rigoo}) involves at least $k$ elements $a_i$, for some $i$, then every one of these products is equal to zero by Proposition \ref{lindass}, Remark \ref{juradpr} and the $(\Sigma, \Delta)$-compatibility of $R$ (more exactly, Proposition \ref{ReyesSuarezUMA2017Prop3.8}). In this way, every term of $f^{(m+1){k}+1}$ is equal to zero, and hence $f\in {\rm nil}(A)$.
\end{proof}
The next theorem generalizes \cite{OuyangChen2010}, Theorem 2.11. We denote ${\rm nil}(R)A:=\{f\in A\mid f= a_0 + a_1X_1 + \dotsb + a_mX_m,\ a_i\in {\rm nil}(R)\}$.
\begin{theorem}\label{2010Theorem2.11}
Let $A$ be a skew PBW extensions over a reversible and $(\Sigma,\Delta)$-compatible ring. If $f=\sum_{i=0}^{m} a_iX_i, g=\sum_{j=0}^{t} b_jY_j$ and $h=\sum_{k=0}^{l}c_kZ_k$ are elements of $A$, and $r$ is any element of $R$, then we have the following assertions:
\begin{enumerate}
\item [\rm (1)] $fg\in {\rm nil}(A) \Leftrightarrow a_ib_j\in {\rm nil}(R)$, for all $i, j$.
\item [\rm (2)] $fgr\in {\rm nil}(A) \Leftrightarrow a_ib_jr\in {\rm nil}(R)$, for all $i, j$.
\item [\rm (3)] $fgh\in {\rm nil}(A)\Leftrightarrow a_ib_jc_k\in {\rm nil}(R)$, for all $i, j, k$.
\end{enumerate}
\end{theorem}
\begin{proof}
(1) As we see in the proof of Theorem \ref{2010Lemma2.10}, ${\rm nil}(A)\subseteq {\rm nil}(R)A$. With this in mind, consider two elements $f, g \in A$ given by $f=\sum_{i=0}^{m} a_iX_i$ and $g=\sum_{j=0}^{t} b_jY_j$ with $fg\in {\rm nil}(A)$. Let $X_i:=x_1^{\alpha_{i1}}\dotsb x_n^{\alpha_{in}}, Y_j:=x_1^{\beta_{j1}}\dotsb x_n^{\beta_{jn}}$, for all $i, j$. We have
\begin{equation*}
fg = \sum_{k=0}^{m+t} \biggl( \sum_{i+j=k} a_iX_ib_jY_j\biggr) \in {\rm nil}(A)\subseteq {\rm nil}(R)A,
\end{equation*}
and ${\rm lc}(fg)= a_m\sigma^{\alpha_m}(b_t)c_{\alpha_m, \beta_t}\in {\rm nil}(R)$. Since the elements $c_{i,j}$ are in the center of $R$, then $c_{\alpha_m,\beta_t}$ are also in the center of $R$, whence $a_m\sigma^{\alpha_m}(b_t)\in {\rm nil}(R)$, and by Lemma \ref{2010Lemma2.8} it follows that $a_mb_t\in {\rm nil}(R)$. The idea is to prove that $a_pb_q\in {\rm nil}(R)$, for $p+q\ge 0$.  We proceed  by induction. Suppose that $a_pb_q\in {\rm nil}(R)$, for $p+q=m+t, m+t-1, m+t-2, \dotsc, k+1$, for some $k>0$. By Lemma \ref{2010Lemma2.7},  we obtain $a_pX_pb_qY_q\in {\rm nil}(R)A$ for these values of $p+q$. In this way, it is sufficient to consider the sum of the products $a_uX_ub_vY_v$, where $u+v=k, k-1,k-2,\dotsc, 0$. Fix $u$ and $v$. Consider the sum of all terms  of $fg$  having exponent $\alpha_u+\beta_v$. By Proposition \ref{lindass}, Remark \ref{juradpr} and the assumption $fg\in {\rm nil}(A)$, we know that the sum of all coefficients of all these terms  can be written as
{\small{\begin{equation}\label{Feder4}
a_u\sigma^{\alpha_u}(b_v)c_{\alpha_u, \beta_v} + \sum_{\alpha_{u'} + \beta_{v'} = \alpha_u + \beta_v} a_{u'}\sigma^{\alpha_{u'}} ({\rm \sigma's\ and\ \delta's\ evaluated\ in}\ b_{v'})c_{\alpha_{u'}, \beta_{v'}}\in {\rm nil}(R).
\end{equation}}}
As we suppose above, $a_pb_q\in {\rm nil}(R)$ for $p+q=m+t, m+t-1, \dotsc, k+1$, so Lemma \ref{2010Lemma2.7} guarantees that the product $a_p(\sigma'$s and $\delta'$s evaluated in $b_q$), for any order of $\sigma'$s and $\delta'$s, is an element of ${\rm nil}(R)$. Since $R$ is reversible, then $({\rm \sigma's\ and\ \delta's\ evaluated\ in}\ b_{q})a_p\in {\rm nil}(R)$. In this way, multiplying (\ref{Feder4}) on the right by $a_k$, and using the fact that the elements $c$'s are in the center of $R$, we obtain that the sum
{\small{\begin{equation}\label{do1234}
a_u\sigma^{\alpha_u}(b_v)a_kc_{\alpha_u, \beta_v} + \sum_{\alpha_{u'} + \beta_{v'} = \alpha_u + \beta_v} a_{u'}\sigma^{\alpha_{u'}} ({\rm \sigma's\ and\ \delta's\ evaluated\ in}\ b_{v'})a_kc_{\alpha_{u'}, \beta_{v'}}
\end{equation}}}
is an element of ${\rm nil}(R)$,  whence, $a_u\sigma^{\alpha_u}(b_0)a_k\in {\rm nil}(R)$. Since $u+v=k$ and $v=0$, then $u=k$, so $a_k\sigma^{\alpha_k}(b_0)a_k\in {\rm nil}(R)$, from which $a_k\sigma^{\alpha_k}(b_0)\in {\rm nil}(R)$ and hence $a_kb_0\in {\rm nil}(R)$ by Lemma \ref{2010Lemma2.8}. Therefore, we now have to study the expression (\ref{Feder4}) for $0\le u \le k-1$ and $u+v=k$. If we multiply (\ref{do1234}) on the right by $a_{k-1}$, then
{\small{\[
a_u\sigma^{\alpha_u}(b_v)a_{k-1}c_{\alpha_u, \beta_v} + \sum_{\alpha_{u'} + \beta_{v'} = \alpha_u + \beta_v} a_{u'}\sigma^{\alpha_{u'}} ({\rm \sigma's\ and\ \delta's\ evaluated\ in}\ b_{v'})a_{k-1}c_{\alpha_{u'}, \beta_{v'}}
\]}}
is also an element of ${\rm nil}(R)$. Using a similar reasoning as above, we can see that the element $a_u\sigma^{\alpha_u}(b_1)a_{k-1}c_{\alpha_u, \beta_1}$ belongs to ${\rm nil}(R)$. Since the elements $c$'s are central and left invertible, $a_u\sigma^{\alpha_u}(b_1)a_{k-1}\in {\rm nil}(R)$, and using the fact $u=k-1$, we have $a_{k-1}\sigma^{\alpha_{k-1}}(b_1)\in {\rm nil}(R)$, from which $a_{k-1}b_1\in {\rm nil}(R)$. Continuing in this way we prove that $a_ib_j\in {\rm nil}(R)$, for $i+j=k$. Therefore $a_ib_j\in {\rm nil}(R)$, for $0\le i\le m$ and $0\le j\le t$.\\

Conversely, for the elements $f, g$ above, suppose that $a_ib_j\in {\rm nil}(R)$. From Lemma \ref{2010Lemma2.7}  we know that $a\sigma^{\alpha}(\delta^{\beta}(b))$ and $a\delta^{\beta}(\sigma^{\alpha}(b))$ are elements of ${\rm nil}(R)$, for every $\alpha,\beta\in \mathbb{N}^{n}$. Now, having in mind that for every product of the form $a_iX_ib_jY_j$, where
$X_i:=x_1^{\alpha_{i1}}\dotsb x_n^{\alpha_{in}}$ and $Y_j:=x_1^{\beta_{j1}}\dotsb x_n^{\beta_{jn}}$, we have the following equality
\begin{align*}
a_iX_ib_jY_j = &\ a_i\sigma^{\alpha_i}(b_j)x^{\alpha_i}x^{\beta_j} + a_ip_{\alpha_{i1}, \sigma_{i2}^{\alpha_{i2}}(\dotsb (\sigma_{in}^{\alpha_{in}}(b_j)))} x_2^{\alpha_{i2}}\dotsb x_n^{\alpha_{in}}x^{\beta_j} \\
+ &\ a_i x_1^{\alpha_{i1}}p_{\alpha_{i2}, \sigma_3^{\alpha_{i3}}(\dotsb (\sigma_{{in}}^{\alpha_{in}}(b_j)))} x_3^{\alpha_{i3}}\dotsb x_n^{\alpha_{in}}x^{\beta_j} \\
+ &\ a_i x_1^{\alpha_{i1}}x_2^{\alpha_{i2}}p_{\alpha_{i3}, \sigma_{i4}^{\alpha_{i4}} (\dotsb (\sigma_{in}^{\alpha_{in}}(b_j)))} x_4^{\alpha_{i4}}\dotsb x_n^{\alpha_{in}}x^{\beta_j}\\
+ &\ \dotsb + a_i x_1^{\alpha_{i1}}x_2^{\alpha_{i2}} \dotsb x_{i(n-2)}^{\alpha_{i(n-2)}}p_{\alpha_{i(n-1)}, \sigma_{in}^{\alpha_{in}}(b_j)}x_n^{\alpha_{in}}x^{\beta_j} \\
+ &\ a_i x_1^{\alpha_{i1}}\dotsb x_{i(n-1)}^{\alpha_{i(n-1)}}p_{\alpha_{in}, b_j}x^{\beta_j},
\end{align*}
by Proposition \ref{lindass}, when we compute every summand of $a_iX_ib_jY_j$ we obtain products of the coefficient $a_i$ with several evaluations of $b_j$ in $\sigma$'s and $\delta$'s depending of the coordinates of $\alpha_i$ (Remark \ref{juradpr}), and since $a\sigma_i(\delta^{\beta_i}(b))$ and $a\delta^{\beta_i}(\sigma^{\alpha_i}(b))$ are elements of ${\rm nil}(R)$, then every coefficient of each term of the expansion $fg$ given by
\[
fg = \sum_{k=0}^{m+t}\biggl(\sum_{i+j=k}a_iX_ib_jY_j\biggr),
\]
is an element of ${\rm nil}(R)$. Therefore, Theorem \ref{2010Lemma2.10} implies that the product $fg$ is an element of $R$.\\

(2) Let $g=b_0 + b_1Y_1 + \dotsb + b_tY_t$ be an element of $A$ with $Y_t\succ \dotsb \succ Y_1$. Then
\begin{align*}
gr = &\ (b_0+b_1Y_1 + \dotsb + b_tY_t)r \\
= &\ b_0r + b_1Y_1r + \dotsb + b_tY_tr \\
= &\ b_0r + b_1(\sigma^{\beta_1}(r)Y_1 + p_{\beta_1, r}) + \dotsb +  b_t(\sigma^{\beta_t}(r)Y_t + p_{\beta_t, r})\\
= &\ b_0r + b_1\sigma^{\beta_1}(r)Y_1 + b_1p_{\beta_1, r} + \dotsb + b_t\sigma^{\beta_t}(r)Y_t + b_tp_{\beta_t, r}
\end{align*}
where $p_{\beta_j,r}=0$, or $\deg(p_{\beta_j,r})<|\alpha|$ if
$p_{\beta_j, r}\neq 0$, for $j=1, \dotsc, t$ (Proposition \ref{coefficientes}). Note that ${\rm lc}(gr) = b_t\sigma^{\beta_t}(r)$. Then
{\small{\begin{align*}
fgr = &\ (a_0 + a_1X_1 + \dotsb + a_mX_m)(b_0r + b_1\sigma^{\beta_1}(r)Y_1 + b_1p_{\beta_1, r} + \dotsb + b_t\sigma^{\beta_t}(r)Y_t + b_tp_{\beta_t, r})\notag \\
= &\ a_0b_0r + a_0b_1\sigma^{\beta_1}(r)Y_1 + a_0b_1p_{\beta_1,r} + \dotsb + a_0b_t\sigma^{\beta_t}(r)Y_t + a_0b_tp_{\beta_t,r} \notag \\
+ &\ a_1X_1b_0r + a_1X_1b_1\sigma^{\beta_1}(r)Y_1 + a_1X_1b_1p_{\beta_1,r} + \dotsb + a_1X_1b_t\sigma^{\beta_t}(r)Y_t + a_1X_1b_tp_{\beta_t,r}\\
+ &\ \dotsb + a_mX_mb_0r + a_mX_mb_1\sigma^{\beta_1}(r)Y_1 + a_mX_mb_1p_{\beta_1,r} + \dotsb + a_mX_mb_t\sigma^{\beta_t}(r)Y_t \\
+ &\ a_mX_mb_tp_{\beta_t,r}\notag \\
= &\ a_0b_0r + a_0b_1\sigma^{\beta_1}(r)Y_1 + a_0b_1p_{\beta_1,r} + \dotsb + a_0b_t\sigma^{\beta_t}(r)Y_t + a_0b_tp_{\beta_t,r}\notag \\
+ &\ a_1[\sigma^{\alpha_1}(b_0r)X_1 + p_{\alpha_1,b_0r}] + a_1[\sigma^{\alpha_1}(b_1\sigma^{\beta_1}(r))X_1 + p_{\alpha_1, b_1\sigma^{\beta_1}(r)}]Y_1 \\
+ &\ a_1[\sigma^{\alpha_1}(b_1) + p_{\alpha_1,b_1}]p_{\beta_1,r} + \dotsb + a_1[\sigma^{\alpha_1}(b_t\sigma^{\beta_t}(r))X_1 + p_{\alpha_1,b_t\sigma^{\beta_t}(r)}]Y_t\\
+ &\ a_1[\sigma^{\alpha_1}(b_t)X_1 + p_{\alpha_1,b_t}]p_{\beta_t, r} + \dotsb + a_m[\sigma^{\alpha_m}(b_0r) + p_{\alpha_m,b_0r}] \\
+ &\ a_m[\sigma^{\alpha_m}(b_1\sigma^{\beta_1}(r)X_1 + p_{\alpha_1, b_1\sigma^{\beta_1}(r)}]Y_1 + a_m[\sigma^{\alpha_m}(b_1)X_m + p_{\alpha_m,b_1}]p_{\beta_1,r}\\
+ &\ \dotsb + a_m[\sigma^{\alpha_m}(b_t\sigma^{\beta_t}(r))X_m + p_{\alpha_m,b_t\sigma^{\beta_t}(r)}]Y_t + a_m[\sigma^{\alpha_m}(b_t)X_m + p_{\alpha_m, b_t}]p_{\beta_t,r},
\end{align*}}}
whence ${\rm lc}(fgr) = a_m\sigma^{\alpha_m}(b_t\sigma^{\beta_t}(r))$, and since $R$ is $\Sigma$-compatible, Lemma \ref{2010Lemma2.8} implies that $a_mb_tr\in {\rm nil}(R)$. Now, Lemma \ref{2010Lemma2.7} guarantees that every term of any polynomial containing the product $a_mb_tr$ in the expression above for $fgr$ is an element of ${\rm nil}(R)A$. In this way, using an monomial order we can repeat this argument for the next monomial of $fgr$ less than ${\rm lc}(fgr)$, and continuing this process until the first monomial to obtain that the elements $a_ib_jr$ are in $\in {\rm nil}(R)$, for all $i, j$.

Conversely, suppose that $a_ib_jr\in {\rm nil}(R)$, for every $i, j$, as above. As we saw above,
\begin{align*}
gr = &\ (b_0+b_1Y_1 + \dotsb + b_tY_t)r \\
= &\ b_0r + b_1Y_1r + \dotsb + b_tY_tr \\
= &\ b_0r + b_1(\sigma^{\beta_1}(r)Y_1 + p_{\beta_1, r}) + \dotsb +  b_t(\sigma^{\beta_t}(r)Y_t + p_{\beta_t, r})\\
= &\ b_0r + b_1\sigma^{\beta_1}(r)Y_1 + b_1p_{\beta_1, r} + \dotsb + b_t\sigma^{\beta_t}(r)Y_t + b_tp_{\beta_t, r}
\end{align*}
where $p_{\beta_j,r}=0$, or $\deg(p_{\beta_j,r})<|\alpha|$ if
$p_{\beta_j, r}\neq 0$, for $j=1, \dotsc, t$. Since $a_ib_jr\in {\rm nil}(R)$, for every $i, j$, Lemma \ref{2010Lemma2.7} implies that $a_ib_j\sigma^{\alpha}(\delta^{\beta}(r))$ and $a_ib_j\delta^{\beta}(\sigma^{\alpha}(r))$ are elements of ${\rm nil}(R)$, for every $\alpha, \beta\in \mathbb{N}^{n}$. In this way, Proposition \ref{lindass} and Remark \ref{juradpr} applied to expression above for the product $fgr$ imply that every one of these summands have coefficients in ${\rm nil}(R)$, and since ${\rm nil}(R)$ is an ideal of $R$ because $R$ is reversible, Theorem \ref{2010Lemma2.7} shows that $fgr\in {\rm nil}(A)$. \\

(3) The equivalence follows from (1) and (2) considering the product $gh$ as the only element $p\in A$.
\end{proof}
\begin{remark}
About Theorem \ref{2010Theorem2.11} (1) we have the following two important observations. (a) In \cite{Reyes2018}, Definition 3.1, the first author introduced the skew $\Pi$-Armendariz rings in the following way: If $A$ is a skew PBW extension over a ring $R$, then $R$ is called a {\em skew}-$\Pi$ {\em Armendariz ring}, if for elements $f=\sum_{i=0}^{m} a_iX_i,\ g=\sum_{j=0}^{t} b_jY_j$ of $A$, $fg\in {\rm nil}(A)$ implies that $a_ib_j\in {\rm nil}(R)$, for every $0\le i\le m$ and $0\le j\le t$. The importance of Theorem \ref{2010Theorem2.11} is explicited, since we are proving in this theorem that skew PBW extensions over skew $\Pi$-Armendariz rings are contained in skew PBW extensions over reversible and $(\Sigma,\Delta)$-compatible rings. (b) In \cite{ReyesSuarezUMA2018}, Definition 4.1, the authors introduced the condition (SA1): if $A$ is a skew PBW extension of $R$, we say that $R$ satisfies the condition (SA1), if whenever $fg=0$ for $f=a_0+a_1X_1+\dotsb + a_mX_m$ and $g=b_0 + b_1Y_1 + \dotsb + b_tY_t$ elements of $A$, then $a_ib_j = 0$, for every $i, j$. It is clear that Theorem \ref{2010Theorem2.11} extends this condition.
\end{remark}
The next theorem extends \cite{OuyangChen2010}, Theorem 2.12.
\begin{theorem}\label{2010Theorem 2.12}
If $A$ is a skew PBW extension over a reversible and $(\Sigma,\Delta)$-compatible ring, then $R$ is weak symmetric if and only if $A$ is weak symmetric.
\end{theorem}
\begin{proof}
Having in mind that a subring of a weak symmetric ring is also a weak symmetric ring, we will only prove one implication. Suppose that $R$ is a weak symmetric ring. If  $f=\sum_{i=0}^{s} a_iX_i, g=\sum_{j=0}^{t} b_jY_j$ and $h=\sum_{k=0}^{l}c_kZ_k$ are elements of $A$ with $fgh\in {\rm nil}(A)$, then Theorem \ref{2010Theorem2.11} implies that $a_ib_jc_k\in {\rm nil}(r)$, for every $i, j, k$, and hence $a_ic_kb_j\in {\rm nil}(R)$, for each $i, j, k$, since $R$ is weak symmetric. Finally, Theorem \ref{2010Theorem2.11} shows that $fhg\in {\rm nil}(A)$.
\end{proof}
\begin{corollary}
If $R$ is a $\Sigma$-rigid ring, then $R$ is weak symmetric if and only if $A$ is weak symmetric.
\end{corollary}
\begin{proof}
Since we have the implications reduced $\Rightarrow$  symmetric $\Rightarrow$ weak symmetric, then the assertion follows from Theorem \ref{2010Theorem 2.12}.
\end{proof}
\begin{corollary}[\cite{OuyangChen2010}, Corollaries 2.13 and 2.14]
Let $B$ be a reversible ring. Then we have the following:
\begin{enumerate}
\item [\rm (1)] $B$ is weak symmetric if and only if $B[x]$ is weak symmetric.
\item [\rm (2)] If $B$ is $\sigma$-compatible, then $B$ is weak symmetric if and only if $B[x;\sigma]$ is weak symmetric.
\item [\rm (3)] If $B$ is $\delta$-compatible, then $B$ is weak symmetric if and only if the differential polynomial ring $B[x;\sigma]$ is weak symmetric.
\item [\rm (4)] Let $\alpha$ be an endomorphism and $\delta$ and $\alpha$-derivation of $R$. If $R$ is $\alpha$-rigid, then $R$ is weak symmetric if and only if $R[x;\alpha,\delta]$ is weak symmetric.
\end{enumerate}
\end{corollary}
With the aim of establishing Theorems \ref{2010Theorem2.17} and  \ref{2010Theorem 3.9}, we need to formulate a criterion which allows us to extend the family $\Sigma$ of injective endomorphisms, and the family of $\Sigma$-derivations $\Delta$ of the ring $R$ to the ring $A$. For the next proposition consider the injective endomorphisms $\sigma_i\in \Sigma$, and the $\sigma_i$-derivations $\delta_i\in \Delta$ $(1\le i\le n)$ formulated in Proposition \ref{sigmadefinition} (compare with \cite{Artamonov2015} where the derivations of skew PBW extensions were computed partially). We include its proof with the objective of appreciating the importance of the assumptions established in the result.
\begin{proposition}[\cite{ReyesSuarezClifford2017}, Theorem 5.1]\label{ReyesSuarez2017CliffordTheorem5.1}
Let $A$ be a skew PBW  extension of a ring $R$. Suppose that $\sigma_i\delta_j=\delta_j\sigma_i,\ \delta_i\delta_j=\delta_j\delta_i$, and $\delta_k(c_{i,j}) = \delta_k(r_l^{(i,j)}) = 0$, for $1\le i, j, l\le n$, where $c_{i,j}$ and $r_l^{(i,j)}$ are as in Definition \ref{gpbwextension}. If $\overline{\sigma_{k}}:A\to A$ and $\overline{\delta_k}:A\to A$ are the functions given by $\overline{\sigma_{k}}(f):=\sigma_k(a_0)+\sigma_k(a_1)X_1 + \dotsb + \sigma_k(a_m)X_m$ and $\overline{\delta_k}(f):=\delta_k(a_0) + \delta_k(a_1)X_1 + \dotsb + \delta_k(a_m)X_m$, for every $f=a_0 + a_1X_1+\dotsb + a_mX_m\in A$, respectively, and $\overline{\sigma_k}(r):=\sigma_i(k)$, for every $1\le i\le n$, then $\overline{\sigma_k}$ is an injective endomorphism of $A$ and $\overline{\delta_k}$ is a $\overline{\sigma_k}$-derivation of $A$. Let $\overline{\Sigma}:=\{\overline{\sigma_1},\dotsc, \overline{\sigma_n}\}
$ and $\overline{\Delta}:=\{\overline{\delta_1},\dotsc, \overline{\delta_n}\}$.
\end{proposition}
\begin{proof}
It is clear that $\overline{\sigma_i}$ is an injective endomorphism of $A$, and that $\overline{\delta_i}$ is an additive map of $A$, for every $1\le i\le n$. Next, we show that $\overline{\delta_i}(fg)=\overline{\sigma_i}(f)\overline{\delta_i}(g) + \overline{\delta_i}(f)g$, for $f,g\in A$.

Consider the elements $f=a_0 + a_1X_1+a_2X_2+\dotsb + a_mX_m$ and $g=b_0 + b_1Y_1+b_2Y_2+\dotsb + b_tY_t$. Since $\overline{\sigma_k}$ and $\overline{\delta_k}$ are additive, for every $i$, it is enough to show that
\begin{align}\label{serpi}
\overline{\delta_k}(a_iX_ib_jY_j) = \overline{\sigma_k}(a_iX_i)\overline{\delta_k}(b_jY_j) + \overline{\delta_k}(a_iX_i)b_jY_j,
\end{align}
for every $1\le i, j \le n$. As an illustration of the necessity of the assumptions above, consider the next particular computations:
{\small{\begin{align}
\overline{\delta_k}(bx_jax_i) = &\ \overline{\delta_k}b(\sigma_j(a)x_j + \delta_j(a))x_i) = \overline{\delta_k}(b\sigma_j(a)x_jx_i + b\delta_j(a)x_1)\notag \\
= &\ \overline{\delta_k}\biggl(b\sigma_j(a)\biggl(c_{i,j}x_ix_j + r_0 + \sum_{l=1}^{n}r_lx_l\biggr) + b\delta_j(a)x_i\biggr)\notag \\
= &\ \overline{\delta_k}\biggl(b\sigma_j(a)c_{i,j}x_ix_j + b\sigma_j(a)r_0 + b\sigma_j(a)\sum_{l=1}^{n}r_lx_l + b\delta_j(a)x_i\biggr) \notag \\
= &\ \delta_k(b\sigma_j(a)c_{i,j})x_ix_j + \delta_k(b\sigma_j(a)r_0) + \sum_{l=1}^{n}\delta_k(b\sigma_j(a)r_l)x_l + \delta_k(b\delta_j(a))x_i\notag
\end{align}}}
or what is the same,
{\small{\begin{align}
\overline{\delta_k}(bx_jax_i) = &\ \sigma_k(b\sigma_j(a))\delta_j(c_{i,j})x_ix_j + \delta_k(b\sigma_j(a))c_{i,j}x_ix_j + \sigma_k(b\sigma_j(a))\delta_i(r_0) \notag \\
+ &\ \delta_k(b\sigma_j(a))r_0
+ \sum_{l=1}^{n}\sigma_k(b\sigma_j(a))\delta_i(r_l)x_l + \sum_{l=1}^{n} \delta_k(b\sigma_j(a))r_lx_l \notag \\
+ &\ \sigma_k(b)\delta_k(\delta_j(a))x_i + \delta_k(b)\delta_j(a)x_i\notag \\
&\ \sigma_k(b)\sigma_k(\sigma_j(a))\delta_j(c_{i,j})x_ix_j + \sigma_k(b)\delta_k(\sigma_j(a))c_{i,j}x_ix_j + \delta_k(b)\sigma_j(a)c_{i,j}x_ix_j \notag \\
+ &\ \sigma_k(b)\sigma_k(\sigma_j(a))\delta_i(r_0) + \sigma_k(b)\delta_k(\sigma_j(a))r_0 + \delta_k(b)\sigma_j(a)r_0 \notag \\
+ &\ \sum_{l=1}^{n} \sigma_k(b)\sigma_k(\sigma_j(a))\delta_i(r_l)x_l + \sum_{l=1}^{n} \sigma_k(b)\delta_k(\sigma_j(a))r_lx_l +\sum_{l=1}^{n}\delta_k(b)\sigma_j(a)r_lx_l \notag \\
+ &\ \sigma_k(b)\delta_k(\delta_j(a))x_i + \delta_k(b)\delta_j(a)x_i.\label{colooo}
\end{align}}}
On the other hand,
{\small{\begin{align}
\overline{\sigma_k}(bx_j)\overline{\delta_k}(ax_i) + \overline{\delta_k}(bx_j)ax_i = &\ \sigma_k(b)x_j \delta_k(a)x_i + \delta_k(b)x_jax_i\notag \\
= &\ \sigma_k(b)(\sigma_j(\delta_k(a))x_j + \delta_j(\delta_k(a)))x_i \notag \\
+ &\ \delta_k(b)(\sigma_j(a)x_j + \delta_j(a))x_i\notag\\
= &\ \sigma_k(b)\sigma_j(\delta_k(a))x_jx_i + \sigma_k(b)\delta_j(\delta_k(a))x_i \notag \\
+ &\ \delta_k(b)\sigma_j(a)x_jx_i + \delta_k(b)\delta_j(a)x_i\notag \\
= &\ \sigma_k(b)\sigma_j(\delta_k(a))\biggl(c_{i,j}x_ix_j + r_0 + \sum_{l=1}^{n}r_lx_l\biggr)\notag \\
+ &\ \sigma_k(b)\delta_j(\delta_k(a))x_i \notag \\
+ &\ \delta_k(b)\sigma_j(a)\biggl(c_{i,j}x_ix_j + r_0 +
\sum_{l=1}^{n} r_lx_l\biggr) + \delta_k(b)\delta_j(a)x_i\notag\\
 = &\ \sigma_k(b)\sigma_j(\delta_k(a))c_{i,j}x_ix_j + \sigma_k(b)\sigma_j(\delta_k(a))r_0 \notag \\
+ &\ \sigma_k(b)\sigma_j(\delta_k(a))\sum_{l=1}^{n} r_lx_l + \sigma_k(b)\delta_j(\delta_k(a))x_i\notag \\
+ &\  \delta_k(b)\sigma_j(a)c_{i,j}x_ix_j + \delta_k(b)\sigma_j(a)r_0 \notag \\
+ &\ \delta_k(b)\sigma_j(a)\sum_{l=1}^{n}r_lx_l + \delta_k(b)\delta_j(a)x_i.\label{coloooo}
\end{align}}}
If we want that the expressions (\ref{colooo}) and (\ref{coloooo}) represent the same value, that is,
{\small{\[
\overline{\delta_k}(bx_jax_i) =  \overline{\sigma_k}(bx_j)\overline{\delta_k}(ax_i) + \overline{\delta_k}(bx_j)ax_i,\ \ \ \ \ 1\le i, j, k, \le n
\]}}
then we have to impose that $\sigma_i\delta_j = \delta_j\sigma_i$, $\delta_i\delta_j = \delta_j\delta_i$,  $\delta_k(c_{i,j}) = \delta_k(r_l^{(i,j)}) = 0$, for $1\le i, j, l\le n$, where $c_{i,j}$ and $r_l^{(i,j)}$ are the elements established in Definition
\ref{gpbwextension}. This justifies the assumptions in the proposition.

Now, the proof of the general case, that is, the expression (\ref{serpi}), it follows from the above reasoning and Remark \ref{juradpr}. Let us see the details. Consider the following expressions:
{\small{\begin{align}
\overline{\delta_k}(a_iX_ib_jY_j) = &\ \overline{\delta_k}(a_i(\sigma^{\alpha_i}(b_j)X_i + p_{\alpha_i, b_j})Y_j)=\overline{\delta_k}(a_i\sigma^{\alpha_i}(b_j)X_iY_j + a_ip_{\alpha_i, b_j}Y_j)\notag \\
= &\ \overline{\delta_k}(a_i\sigma^{\alpha_i}(b_j)(c_{\alpha_i, \beta_j}x^{\alpha_i+\beta_j} + p_{\alpha_i,\beta_j}) + a_ip_{\alpha_i,b_j}Y_j)\notag \\
= &\ \overline{\delta_k}(a_i\sigma^{\alpha_i}(b_j)c_{\alpha_i,\beta_j}x^{\alpha_i+\beta_j} + a_i\sigma^{\alpha_i}(b_j)p_{\alpha_i, \beta_j} + a_ip_{\alpha_i,b_j}Y_j)\notag \\
= &\ \overline{\delta_k}(a_i\sigma^{\alpha_i}(b_j)c_{\alpha_i,\beta_j})x^{\alpha_i+\beta_j} + \overline{\delta_k}(a_i\sigma^{\alpha_i}(b_j)p_{\alpha_i,\beta_j}) + \overline{\delta_k}(a_ip_{\alpha_i,b_j}Y_j)\notag\\
= &\ \sigma_k(a_i\sigma^{\alpha_i}(b_j))\delta_k(c_{\alpha_i,\beta_j})x^{\alpha_i+\beta_j} + \delta_k(a_i\sigma^{\alpha_i}(b_j))c_{\alpha_i,\beta_j}x^{\alpha_i+\beta_j}\notag \\
+ &\ \overline{\delta_k}(a_i\sigma^{\alpha_i}(b_j)p_{\alpha_i,\beta_j}) + \overline{\delta_k}(a_ip_{\alpha_i,b_j}Y_j)\notag \\
= &\ \sigma_k(a_i)\sigma_k(\sigma^{\alpha_i}(b_j))\delta_k(c_{\alpha_i,\beta_j})x^{\alpha_i+\beta_j} + \sigma_k(a_i)\delta_k(\sigma^{\alpha_i}(b_j))c_{\alpha_i,\beta_j}x^{\alpha_i+\beta_j}\notag \\
+ &\ \delta_k(a_i)\sigma^{\alpha_i}(b_j)c_{\alpha_i,\beta_j}x^{\alpha_i+\beta_j} + \overline{\delta_k}(a_i\sigma^{\alpha_i}(b_j)p_{\alpha_i,\beta_j}) + \overline{\delta_k}(a_ip_{\alpha_i,b_j}Y_j), \notag
\end{align}}}
and
{\small{\begin{align}
\overline{\sigma_k}(a_iX_i)\overline{\delta_k}(b_jY_j) + \overline{\delta_k}(a_iX_i)b_jY_j = &\ \sigma_k(a_i)X_i\delta_k(b_j)Y_j + \delta_k(a_i)X_ib_jY_j\notag \\
= &\ \sigma_k(a_i)(\sigma^{\alpha_i}(\delta_k(b_j))X_i + p_{\alpha_i,\delta_k(b_j)})Y_j \notag \\
+ &\ \delta_k(a_i)(\sigma^{\alpha_i}(b_j)X_i + p_{\alpha_i, b_j})Y_j\notag \\
= &\ \sigma_k(a_i)(\sigma^{\alpha_i}(\delta_k(b_j))X_iY_j) + \sigma_k(a_i)p_{\alpha_i, \delta_k(b_j)}Y_j\notag \\
+ &\ \delta_k(a_i)\sigma^{\alpha_i}(b_j)X_iY_j + \delta_k(a_i)p_{\alpha_i,b_j}Y_j\notag \\
= &\ \sigma_k(a_i)\sigma^{\alpha_i}(\delta_k(b_j))(c_{\alpha_i,\beta_j}x^{\alpha_i+\beta_j} + p_{\alpha_i,\beta_j}) \notag \\
+ &\ \sigma_k(a_i)p_{\alpha_i,\delta_k(b_j)}Y_j\notag \\
+ &\ \delta_k(a_i)\sigma^{\alpha_i}(b_j)(c_{\alpha_i,\beta_j}x^{\alpha_i+\beta_j} + p_{\alpha_i,\beta_j})\notag \\
+ &\ \delta_k(a_i)p_{\alpha_i,b_j}Y_j\notag \\
= &\ \sigma_k(a_i)\sigma^{\alpha_i}(\delta_k(b_j))c_{\alpha_i,\beta_j}x^{\alpha_i+\beta_j} \notag \\
+ &\ \sigma_k(a_i)\sigma^{\alpha_i}(\delta_k(b_j))p_{\alpha_i, \beta_j} \notag \\
+ &\ \sigma_k(a_i)p_{\alpha_i, \delta_k(b_j)}Y_j + \delta_k(a_i)\sigma^{\alpha_i}(b_j)c_{\alpha_i, \beta_j}x^{\alpha_i+\beta_j}\notag \\
+ &\ \delta_k(a_i)\sigma^{\alpha_i}(b_j)p_{\alpha_i, \beta_j} +
\delta_k(a_i)p_{\alpha_i, b_j}Y_j.\notag
\end{align}}}
By assumption, we have the equalities $\sigma_k(a_i)\delta_k(\sigma^{\alpha_i}(b_j)) = \sigma_k(a_i)\sigma^{\alpha_i}(\delta_k(b_j))$ and $\delta_k(c_{\alpha_i,\beta_j}) = 0$, which means that we need to prove the relation
{\small{\begin{align}
\overline{\delta_k}(a_i\sigma^{\alpha_i}(b_j)p_{\alpha_i,\beta_j}) + \overline{\delta_k}(a_ip_{\alpha_i,b_j}Y_j) = &\
 \sigma_k(a_i)\sigma^{\alpha_i}(\delta_k(b_j))p_{\alpha_i, \beta_j} + \sigma_k(a_i)p_{\alpha_i, \delta_k(b_j)}Y_j \notag \\
+ &\ \delta_k(a_i)\sigma^{\alpha_i}(b_j)p_{\alpha_i, \beta_j} + \delta_k(a_i)p_{\alpha_i, b_j}Y_j.\label{metal}
\end{align}}}
However, note that this equality is a consequence of the linearity of $\delta_k$, Remark \ref{juradpr}, and the assumptions established in the formulation of the theorem. More precisely, using these facts we have
{\normalsize{\begin{align}
\overline{\delta}_k(a_i\sigma^{\alpha_i}(b_j)p_{\alpha_i, \beta_j}) = &\ \overline{\sigma_k}(a_i\sigma^{\alpha_i}(b_j))\overline{\delta_k}(p_{\alpha_i,\beta_j}) + \overline{\delta_k}(a_i\sigma^{\alpha_i}(b_j))p_{\alpha_i, \beta_j}\notag \\
= &\ \sigma_k(a_i)\sigma_k(\sigma^{\alpha_i}(b_j))\overline{\delta_k}(p_{\alpha_i,\beta_j}) + \sigma_k(a_i) \delta_k(\sigma^{\alpha_i}(b_j))p_{\alpha_i, \beta_j} \notag \\
+ &\ \delta_k(a_i)\sigma^{\alpha_i}(b_j)p_{\alpha_i,\beta_j}\notag \\
= &\ \sigma_k(a_i) \delta_k(\sigma^{\alpha_i}(b_j))p_{\alpha_i, \beta_j} + \delta_k(a_i)\sigma^{\alpha_i}(b_j)p_{\alpha_i,\beta_j}\notag \\
= &\ \sigma_k(a_i)\sigma^{\alpha_i}(\delta_k(b_j))p_{\alpha_i,\beta_j} + \delta_k(a_i)\sigma^{\alpha_i}(b_j)p_{\alpha_i, \beta_j},\label{acdc}
\end{align}}}
and,
{\normalsize{\begin{align}
\overline{\delta_k}(a_ip_{\alpha_i,b_j}Y_j) = &\ \overline{\sigma_k}(a_i)\overline{\delta_k}(p_{\alpha_i,b_j}Y_j) + \overline{\delta_k}(a_i)p_{\alpha_i,b_j}Y_j\notag \\
= &\ \sigma_k(a_i)p_{\alpha_i, \delta_k(b_j)}Y_j + \delta_k(a_i)p_{\alpha_i,b_j}Y_j,\label{marliiii}
\end{align}}}
where we can see that expression (\ref{metal}) is precisely the sum of (\ref{acdc}) and (\ref{marliiii}). Therefore $\overline{\delta_i}$ is a $\overline{\sigma_i}$-derivation of $A$.
\end{proof}
With Proposition \ref{ReyesSuarez2017CliffordTheorem5.1} in our hands, we formulate Theorem \ref{2010Theorem2.17} which extends \cite{OuyangChen2010}, Theorem 2.17. To this end, consider the skew PBW extension $A'$ induced by injective endomorphisms and derivations established in Proposition \ref{ReyesSuarez2017CliffordTheorem5.1}, i.e., $A' = \sigma(A)\langle x_1',\dotsc, x_n'\rangle$. We remark that using algorithms established by Reyes and Su\'arez (2017b) one can prove that $A'$ is a left free $A$-module considering adequate relations between the indeterminates $x_1',\dotsc, x_n'$. For the sets of injective endomorphisms $\overline{\Sigma}$ and $\overline{\Sigma}$-derivations $\overline{\Delta}$ formulated in Proposition \ref{ReyesSuarez2017CliffordTheorem5.1}, consider a definition of $(\overline{\Sigma},\overline{\Delta})$-compatible in a similar way to the Definition \ref{Definition3.52008}. Suppose that the elements $c_{i,j}$ in Definition \ref{gpbwextension} (iv) are central in $R$, for all $i, j$.
\begin{theorem}\label{2010Theorem2.17}
If $A$ is a skew PBW extension over an $\Sigma$-rigid ring $R$, then $A$ is weak symmetric if and only if $A'$ is weak symmetric.
\end{theorem}
\begin{proof}
As we saw in Section \ref{SigmaDeltaweaksymmetricskewPBWextensions}, if $R$ is $\Sigma$-rigid, then $R$ is reduced, or equivalently, $A$ is reduced whence $A$ is reversible. The aim is to show that $A$ is $(\overline{\Sigma}, \overline{\Delta})$-compatible. From Proposition \ref{ReyesSuarez2018Theorem3.9} we also know that $R$ is $(\Sigma, \Delta)$-compatible.

Consider elements $f=a_0+a_1X_1+\dotsb + a_mX_m,\ g=b_0 + b_1Y_1 + \dotsb + b_tY_t$ in $A$ with $fg=0$ and let us see that $a_ib_j=0$, for every $i, j$. Since
\begin{align*}
fg = &\ (a_0+a_1X_1+\dotsb + a_mX_m)(b_0+b_1Y_1+\dotsb + b_tY_t)\\
= &\ \sum_{k=0}^{m+t} \biggl(\sum_{i+j=k} a_iX_ib_jY_j\biggr),
\end{align*}
then ${\rm lc}(fg)= a_m\sigma^{\alpha_m}(b_t)c_{\alpha_m, \beta_t}=0$ whence $a_m\sigma^{\alpha_m}(b_t)=0$ ($c_{\alpha_m, \beta_b}$ is invertible), and by Proposition \ref{Reyes2015Lemma3.3} (4), $a_mb_t=0$. The idea is to prove that $a_pb_q=0$, for $p+q\ge 0$. We proceed  by induction. Suppose that $a_pb_q=0$, for $p+q=m+t, m+t-1, m+t-2, \dotsc, k+1$, for some $k>0$. By Proposition \ref{Reyes2015Lemma3.3} (5) we obtain $a_pX_pb_qY_q=0$ for these values of $p+q$. In this way we only consider the sum of the products $a_uX_ub_vY_v$, where $u+v=k, k-1,k-2,\dotsc, 0$. Fix $u$ and $v$. Consider the sum of all terms  of $fg$  having exponent $\alpha_u+\beta_v$. By Proposition \ref{lindass}, Remark \ref{juradpr}, and the assumption $fg=0$, the sum of all coefficients of all these terms  can be written as
{\small{\begin{equation}\label{Federer}
a_u\sigma^{\alpha_u}(b_v)c_{\alpha_u, \beta_v} + \sum_{\alpha_{u'} + \beta_{v'} = \alpha_u + \beta_v} a_{u'}\sigma^{\alpha_{u'}} ({\rm \sigma's\ and\ \delta's\ evaluated\ in}\ b_{v'})c_{\alpha_{u'}, \beta_{v'}} = 0.
\end{equation}}}
By assumption we know that $a_pb_q=0$ for $p+q=m+t, m+t-1, \dotsc, k+1$.  So, Proposition \ref{Reyes2015Lemma3.3} (3) guarantees that the product
\[a_p({\rm \sigma's\ and\ \delta's\ evaluated\ in}\ b_{q})\ \ \ \ \ \ \ ({\rm any\  order\ of}\ \sigma's\ {\rm and}\ \delta's)
\]
is equal to zero. Then  $[({\rm \sigma's\ and\ \delta's\ evaluated\ in}\ b_{q})a_p]^2=0$ and hence we obtain the equality $({\rm \sigma's\ and\ \delta's\ evaluated\ in}\ b_{q})a_p=0$ ($R$ is reduced). In this way, multiplying (\ref{Federer}) by $a_k$, and using the fact that the elements $c_{i,j}$ in Definition \ref{gpbwextension} (iv) are in the center of $R$,
{\small{\begin{equation}\label{doooooo}
a_u\sigma^{\alpha_u}(b_v)a_kc_{\alpha_u, \beta_v} + \sum_{\alpha_{u'} + \beta_{v'} = \alpha_u + \beta_v} a_{u'}\sigma^{\alpha_{u'}} ({\rm \sigma's\ and\ \delta's\ evaluated\ in}\ b_{v'})a_kc_{\alpha_{u'}, \beta_{v'}} = 0,
\end{equation}}}
whence, $a_u\sigma^{\alpha_u}(b_0)a_k=0$. Since $u+v=k$ and $v=0$, then $u=k$, so $a_k\sigma^{\alpha_k}(b_0)a_k=0$, i.e., $[a_k\sigma^{\alpha_k}(b_0)]^{2}=0$, from which $a_k\sigma^{\alpha_k}(b_0)=0$ and $a_kb_0=0$ by Proposition \ref{Reyes2015Lemma3.3} (4). Therefore, we now have to study the expression (\ref{Federer}) for $0\le u \le k-1$ and $u+v=k$. If we multiply (\ref{doooooo}) by $a_{k-1}$ we obtain
{\scriptsize{\begin{equation}
a_u\sigma^{\alpha_u}(b_v)a_{k-1}c_{\alpha_u, \beta_v} + \sum_{\alpha_{u'} + \beta_{v'} = \alpha_u + \beta_v} a_{u'}\sigma^{\alpha_{u'}} ({\rm \sigma's\ and\ \delta's\ evaluated\ in}\ b_{v'})a_{k-1}c_{\alpha_{u'}, \beta_{v'}} = 0.
\end{equation}}}
Using a similar reasoning as above, we can see that $a_u\sigma^{\alpha_u}(b_1)a_{k-1}c_{\alpha_u, \beta_1}=0$. Since $A$ is bijective, $a_u\sigma^{\alpha_u}(b_1)a_{k-1}=0$, and using the fact $u=k-1$, we have $[a_{k-1}\sigma^{\alpha_{k-1}}(b_1)]=0$, which imply $a_{k-1}\sigma^{\alpha_{k-1}}(b_1)=0$, that is, $a_{k-1}b_1=0$. Continuing in this way we prove that $a_ib_j=0$ for $i+j=k$. Hence $a_ib_j=0$, for $0\le i\le m$ and $0\le j\le t$, and therefore $a_i\sigma^{\alpha}(b_j)) = a_i\delta^{\beta}(b_j)=0$, for all $\alpha, \beta\in \mathbb{N}^{n}$, since $R$ is $(\Sigma,\Delta)$-compatible. In this way, when we consider the expressions
\begin{align*}
f\overline{\sigma^{\alpha}}(g) = &\ (a_0+a_1X_1 + \dotsb + a_mX_m)({\sigma^{\alpha}}(b_0) + {\sigma^{\alpha}}(b_1)Y_1 + \dotsb + {\sigma^{\alpha}}(b_t)Y_t)\\
= &\ \sum_{k=0}^{m+t} \biggl(\sum_{i+j=k} a_iX_i\sigma^{\alpha}(b_j)Y_j\biggr)\\
= &\ \sum_{k=0}^{m+t} \biggl(\sum_{i+j=k} a_i[\sigma^{\alpha_i}(\sigma^{\alpha}(b_j))X_i + p_{\alpha_i, \sigma^{\alpha}(b_j)}]Y_j\biggr)\\
= &\ \sum_{k=0}^{m+t} \biggl(\sum_{i+j=k} a_i\sigma^{\alpha_i}(\sigma^{\alpha}(b_j))X_iY_j + a_ip_{\alpha_i, \sigma^{\alpha}(b_j)}Y_j\biggr)\\
= &\ \sum_{k=0}^{m+t} \biggl(\sum_{i+j=k} a_i\sigma^{\alpha_i}(\sigma^{\alpha}(b_j))[c_{\alpha_i,\beta_j}x^{\alpha_i + \beta_j} + p_{\alpha_i,\beta_j}] + a_ip_{\alpha_i, \sigma^{\alpha}(b_j)}Y_j\biggr)
\end{align*}
and
\begin{align*}
f\overline{\delta^{\beta}}(g) = &\ (a_0+a_1X_1 + \dotsb + a_mX_m)({\delta^{\beta}}(b_0) + {\delta^{\beta}}(b_1)Y_1 + \dotsb + {\delta^{\beta}}(b_t)Y_t)\\
= &\ \sum_{k=0}^{m+t} \biggl(\sum_{i+j=k} a_iX_i\delta^{\beta}(b_j)Y_j\biggr)\\
= &\ \sum_{k=0}^{m+t} \biggl(\sum_{i+j=k} a_i[\sigma^{\alpha_i}(\delta^{\beta}(b_j))X_i + p_{\alpha_i, \delta^{\beta}(b_j)}]Y_j\biggr)\\
= &\ \sum_{k=0}^{m+t} \biggl(\sum_{i+j=k} a_i\sigma^{\alpha_i}(\delta^{\beta}(b_j))X_iY_j + a_ip_{\alpha_i, \delta^{\beta}(b_j)}Y_j\biggr)
\end{align*}
or equivalently,
\begin{align*}
f\overline{\delta^{\beta}}(g) = &\ \sum_{k=0}^{m+t} \biggl(\sum_{i+j=k} a_i\sigma^{\alpha_i}(\delta^{\beta}(b_j))[c_{\alpha_i,\beta_j}x^{\alpha_i + \beta_j} + p_{\alpha_i,\beta_j}] + a_ip_{\alpha_i, \delta^{\beta}(b_j)}Y_j\biggr),
\end{align*}
Proposition \ref{lindass} and Remark \ref{juradpr} imply that $f\overline{\sigma^{\alpha}}(g) = f\overline{\delta^{\beta}}(g) = 0$, for every $\alpha,\beta\in \mathbb{N}^{n}$. In a similar way, if we start with the equality $f\overline{\sigma^{\alpha}}(g) =0$, then we can show that $fg=0$, which means that $A$ is $(\Sigma,\Delta)$-compatible. In this way, since we have showed that $A$ is reversible and $(\overline{\Sigma}, \overline{\Delta})$-compatible, the assertion which we are proving it follows from Theorem \ref{2010Theorem 2.12}.
\end{proof}
\section{Skew PBW extensions over weak $(\Sigma,\Delta)$-symmetric rings}\label{SigmaDeltaweaksymmetricskewPBWextensions}
In \cite{OuyangChen2010}, Definition 2,  Ouyang and Chen 2010 introduced the notion of weak $(\alpha,\delta)$-symmetric ring in the following way: a ring $B$ with an endomorphism $\sigma$ and an $\sigma$-derivation $\delta$ is said to be {\em weak} $\sigma$-{\em symmetric} provided that $abc\in {\rm nil}(B)\Leftrightarrow ac\sigma(b)\in {\rm nil}(B)$, for any elements $a, b, c\in B$. $B$ is said to be {\em weak} $\delta$-symmetric, if for $a, b, c\in R$, $abc\in {\rm nil}(B)$ implies $ac\delta(b)\in {\rm nil}(B)$. If $B$ is both {\em weak} $\sigma$-{\em symmetric} and {\em weak} $\delta$-{\em symmetric}, $B$ is called a {\em weak} $(\Sigma,\Delta)$-{\em symmetric} ring. With respect to the relation between weak symmetric ring and weak $(\alpha,\delta)$-symmetric rings, there is an example of a weak symmetric ring which is not weak $(\alpha,\delta)$-symmetric, see \cite{OuyangChen2010}, Example 3.2. Note that for every subring $S$ of a weak $(\alpha,\delta)$-symmetric ring $B$ which satisfies $\alpha(S)\subseteq S$ and $\delta(S)\subseteq S$, it follows that $S$ is also a weak weak $(\alpha,\delta)$-symmetric ring. With these  definitions in mind, we present in a natural way the notion of weak $(\Sigma,\Delta)$-symmetric ring for a ring $R$ with a family of endomorphisms $\Sigma$ and a family of $\Sigma$-derivations  $\Delta$.
\begin{definition}\label{2010Definition2}
Let $R$ be a ring with a family of endomorphisms of $R$ and a family of $\Sigma=\{\sigma_1,\dotsc,\sigma_n\}$-derivations $\Delta=\{\delta_1,\dotsc,\delta_n\}$. $R$ is called {\em weak} $\Sigma$-{\em symmetric}, if $abc\in {\rm nil}(R)\Rightarrow ac\sigma_i(b)\in {\rm nil}(R)$, for every $i$ and each elements $a, b, c\in R$. $R$ is said to be {\em weak} $\Delta$-{\em symmetric}, if $abc\in {\rm nil}(R) \Rightarrow ac\delta_i(b)\in {\rm nil}(R)$, for every $i$ and each elements $a, b, c\in R$. In the case $R$ is both weak $\Sigma$-symmetric and weak $\Delta$-symmetric, we say that $R$ is a {\em weak} $(\Sigma,\Delta)$-symmetric ring.
\end{definition}
\begin{definition}
If $R$ is a ring with a family of endomorphisms of $R$ and a family of $\Sigma=\{\sigma_1,\dotsc,\sigma_n\}$-derivations $\Delta=\{\delta_1,\dotsc,\delta_n\}$, then an ideal $I$ of $R$ is said to be an {\em weak}-{\em symmetric ideal}, if $abc\in {\rm nil}(R)\Rightarrow ac\sigma_i(b), ac\delta_i(b)\in {\rm nil}(R)$, for each $i$ and every elements $a, b, c\in I$.
\end{definition}
The next proposition extends \cite{OuyangChen2010}, Proposition 3.6.
\begin{proposition}\label{2010Proposition3.6}
If $R$ is an abelian ring with $\sigma_i(e)=e$ and $\delta_i(e)=0$, for any idempotent element $e$ of $R$, then the following statements are equivalent:
\begin{enumerate}
\item [\rm (1)] $R$ is a weak $(\Sigma,\Delta)$-symmetric ring.
\item [\rm (2)] $eR$ and $(1-e)R$ are weak $(\Sigma,\Delta)$-symmetric ideals.
\end{enumerate}
\end{proposition}
\begin{proof}
We use similar arguments to the established in \cite{OuyangChen2010}, Proposition 3.6. $(1)\Rightarrow (2)$ It is clear. $(2)\Rightarrow (1)$ Consider elements $a, b, c\in R$ with $abc\in {\rm nil}(R)$. It follows that $eaebec, (1-e)a(1-e)b(1-e)c\in {\rm nil}(R)$. By assumption, $eR$ and $(1-e)R$ are weak $(\Sigma,\Delta)$-symmetric ideals, so $eaec\sigma_i(eb)=eac\sigma_i(b)\in {\rm nil}(R)$ and $(1-e)a(1-e)c\sigma_i((1-e)b) = (1-e)ac\sigma_i(b)\in {\rm nil}(R)$. This fact shows that $ac\sigma_i(b)\in {\rm nil}(R)$, for every $i$, and hence $R$ is weak $\Sigma$-symmetric. Now, since for any $r\in R$, $\delta_i(er)=\sigma_i(e)\delta_i(r) + \delta_i(e)r = e\delta_i(r)$, for every $i$, the assumptions on $R$ imply that if $abc\in {\rm nil}(R)$, then $ea(eb)(ec), (1-e)a(1-e)b(1-e)c\in {\rm nil}(R)$. Therefore $eaec\delta_i(eb) = eac\delta_i(b), (1-e)a(1-e)c\delta_i((1-e)b) = (1-e)ac\delta_i(b)\in {\rm nil}(R)$. In this way, $ac\delta_i(b)\in {\rm nil}(R)$, for every $i$, which means that  $R$ is weak $\Delta$-symmetric. In conclusion, $R$ is weak $(\Sigma,\Delta)$-symmetric.
\end{proof}
For the next theorem, Theorem \ref{2010Theorem3.7}, we need some preliminary facts and the Proposition \ref{Lezamaetal2015Proposition2.6} which concerns about quotients of skew PBW extensions: consider $A=\sigma(R)\langle x_1,\dotsc, x_n\rangle$ a skew PBW extension of a ring $R$. Let $\Sigma:=\{\sigma_1,\dotsc, \sigma_n\}$ and $\Delta:=\{\delta_1,\dotsc,\delta_n\}$ such as in Proposition \ref{sigmadefinition}. Following \cite{LezamaAcostaReyes2015}, Section 2, if $I$ is an ideal of $R$, $I$ is called $\Sigma$-invariant ($\Delta$-invariant), if it is invariant under each injective endomorphism $\sigma_i$ ($\sigma_i$-derivation $\delta_i$) of $\Sigma$ ($\Delta$), that is, $\sigma_i(I)\subseteq I$ ($\delta_i(I)\subseteq I$), for every $i$. If $I$ is both $\Sigma$ and $\Delta$-invariant ideal, we say that $I$ is $(\Sigma,\Delta)$-invariant.
\begin{proposition}[\cite{LezamaAcostaReyes2015}, Proposition 2.6]\label{Lezamaetal2015Proposition2.6}
 If $A$ is a skew PBW extension over a ring $R$ and $I$ is a $(\Sigma,\Delta)$-invariant ideal of $R$, then the following statements hold:
\begin{enumerate}
\item [\rm (1)] $IA$ is an ideal of $A$ and $IA\cap R=I$. $IA$ is a proper ideal of $A$ if and only if $I$ is proper in $R$. Moreover, if $\sigma_i$ is bijective and $\sigma_i(I)=I$, for every $i$, then $IA=AI$.
\item [\rm (2)] If $I$ is proper and $\sigma_i(I)=$, for every $1\le i\le n$, then $A/IA$ is a skew PBW extension of $R/I$. In fact, if $I$ is proper and $A$ is bijective, then $A/IA$ is a bijective skew PBW extension of $R/I$.
\end{enumerate}
\end{proposition}
From Proposition \ref{Lezamaetal2015Proposition2.6} we can see that if $I$ is $(\Sigma,\Delta)$-invariant, then over $\overline{R}:=R/I$ it is induced a systems $(\overline{\Sigma},\overline{\Delta})$ of endomorphisms $\overline{\Sigma}$ and $\overline{\Sigma}$-derivations $\overline{\Delta}$ defined by $\overline{\sigma_i}(r+I))=\sigma_i(r)+I$ and $\overline{\delta_i}(r+I) = \delta_i(r)+I$, for $1\le i\le n$. We keep the variables $x_1,\dotsc,x_n$ of the extension $A$ to the extension $A/IA$ if no confusion arises. For quotients of skew PBW extensions, we consider the notion of weak $(\Sigma,\Delta)$-symmetric in the natural way following Definition \ref{2010Definition2}.\\

Our next theorem extends \cite{OuyangChen2010}, Theorem 3.7.
\begin{theorem}\label{2010Theorem3.7}
Let $I$ be an $(\Sigma,\Delta)$-invariant and weak $(\Sigma,\Delta)$-symmetric ideal of $R$. If $I\subseteq {\rm nil}(R)$, then $R/I$ is a weak $(\overline{\Sigma}, \overline{\Delta})$-symmetric ring if and only if $R$ is a weak $(\Sigma,\Delta)$-symmetric ring.
\end{theorem}
\begin{proof}
Consider elements $a, b, c\in R$ such that $(a+I)(b+I)(c+I)\in {\rm nil}(R/I)$. There exists a positive integer $m$ with $(abc)^{m}\in I$. Since $I\subseteq {\rm nil}(R)$ it follows that $abc\in {\rm nil}(R)$. By assumption, $R$ is weak $(\Sigma,\Delta)$-symmetric, so $ac\sigma_i(b), ac\delta_i(b) \in {\rm nil}(R)$, for $i=1,\dotsc, n$. Hence $(a+I)(c+I)(\sigma_i(b)+I), (a+I)(c+I)(\delta_i(b)+I) \in {\rm nil}(R/I)$, that is, $(a+I)(c+I)\overline{\sigma_i}(b+I), (a+I)(c+I)\overline{\delta_i}(b+I)\in {\rm nil}(R/I)$. Therefore $R/I$ is weak $(\overline{\Sigma}, \overline{\Delta})$-symmetric.

Conversely, suppose that $R/I$ is a weak $(\Sigma,\Delta)$-symmetric ring. Consider elements $a, b, c\in R$ with $abc\in {\rm nil}(R)$. It is clear that $(a+I)(b+I)(c+I)\in {\rm nil}(R/I)$. Since $R/I$ is weak $(\overline{\Sigma}, \overline{\Delta})$-symmetric, we have that $(a+I)(c+I)(\sigma_i(b) + I) = (ac\sigma_i(b) + I),  (a+I)(c+I)(\delta_i(b) + I) = (ac\delta_i(b) + I) \in {\rm nil}(R/I)$, for $i=1,\dotsc, n$. This means that for every $i$ there exist positive integers $p = p(i), q = q(i)$ depending on $i$, such that $(ac\sigma_i(b))^{p}, (ac\delta_i(b))^{q}\in I$. In this way, $ac\sigma_i(b), ac\delta_i(b)\in I$ because $I\subseteq {\rm nil}(R)$ which shows that $R$ is a weak $(\overline{\Sigma},\overline{\Delta})$-symmetric ring.
\end{proof}
The next theorem generalizes Ouyang and Chen \cite{OuyangChen2010}, Theorem 3.9.
\begin{theorem}\label{2010Theorem 3.9}
If $R$ is a $(\Sigma,\Delta)$-compatible and reversible ring, then $R$ is a weak $(\Sigma,\Delta)$-symmetric ring if and only if $A$ is a weak $(\overline{\Sigma}, \overline{\Delta})$-symmetric ring, where the sets of injective endomorphisms $\overline{\Sigma}$ and $\overline{\Sigma}$-derivations $\overline{\Delta}$ of $A$ are as in Proposition \ref{ReyesSuarez2017CliffordTheorem5.1}.
\end{theorem}
\begin{proof}
If $A$ is a weak $(\overline{\Sigma}, \overline{\Delta})$-symmetric ring, then it is clear that $R$ is weak $(\Sigma,\Delta)$-symmetric ring because $\sigma_i(R), \delta_i(R) \subseteq R$, for every $i=1,\dotsc, n$.

Conversely, suppose that $R$ is weak $(\Sigma,\Delta)$-symmetric ring. Consider the elements $f=\sum_{i=0}^{s} a_iX_i, g=\sum_{j=0}^{t} b_jY_j$ and $h=\sum_{k=0}^{l}c_kZ_k$ of $A$. From Theorem \ref{2010Theorem2.11} we know that $a_ib_jc_k\in {\rm nil}(R)$, for all $i, j, k$, whence $a_ic_k\sigma_l(b_j), a_ic_k\delta_l(b_j)\in {\rm nil}(R)$, for $l=1,\dotsc, n$, since $R$ is weak $(\Sigma,\Delta)$-symmetric. Again, Theorem \ref{2010Theorem2.11} implies that $fh\overline{\sigma_i}(g), fh\overline{\delta_i}(g)\in {\rm nil}(A)$, that is, $A$ is a weak $(\Sigma,\Delta)$-symmetric ring.
\end{proof}
\begin{corollary}[\cite{OuyangChen2010}, Corollary 3.10]
Let $R$ be a reversible ring. Then $R$ is a weak symmetric ring if and only if $R[t]$ is weak symmetric.
\end{corollary}
\section{Examples}\label{examplespaper}
Remarkable examples of skew PBW extensions over $(\Sigma,\Delta)$-compatible and reversible rings can be found in \cite{JaramilloReyes2018, ReyesPhD2013, ReyesYesica2018, SuarezReyesgenerKoszul2017}. In this way, the results obtained in Sections \ref{weaksymmetricskewPBWextensions} and \ref{SigmaDeltaweaksymmetricskewPBWextensions} can be illustrated with every one of these noncommutative rings. Let us just say some of these examples.\\

If $A$ is a skew PBW extension over a ring $R$ where the coefficients commute with the variables, that is, $x_ir = rx_i$, for every $r\in R$ and each $i=1,\dotsc, n$, or equivalently, $\sigma_i = {\rm id}_R$ and $\delta_i = 0$, for every $i$ (these extensions were called {\em constant} in \cite{Suarez2017}, Definition 2.5 (a)), then it is clear that $R$ is a $\Sigma$-rigid ring. Some examples of these  extensions are the following: (i) PBW extensions defined by Bell and Goodearl (which include the classical commutative polynomial rings, universal enveloping algebra of a Lie algebra, and others); some operator algebras (for example, the algebra of linear partial differential operators, the algebra of linear partial shift operators, the algebra of linear partial difference operators, the algebra of linear partial $q$-dilation operators, and the algebra of linear partial q-differential operators). (ii) solvable polynomial rings introduced by Kandri-Rody and Weispfenning. (iii) $G$-algebras introduced by Apel. (iv) PBW algebras defined by Bueso et. al. (v) Calabi-Yau and skew Calabi-Yau algebras. (vi) Koszul and qudratic algebras. $G$-algebras studied by Levandovskyy. (vii) PBW algebras defined by Bueso et al. in \cite{BuesoTorrecillasVerschoren}. A detailed reference of every one of these algebras can be found in \cite{LezamaReyes2014, Suarez2017, SuarezLezamaReyes2017,  SuarezReyesgenerKoszul2017}. Of course, we also encounter examples of skew PBW extensions which are not constant (see \cite{LezamaReyes2014} for the definition of each one of these algebras): the quantum plane $\cO_q(\Bbbk^{2})$; the Jordan plane; the algebra of $q$-differential operators $D_{q,h}[x,y]$; the mixed algebra $D_h$; the operator differential rings; the algebra of differential operators $D_{\bf q}(S_{\bf q})$ on a quantum space ${S_{\bf q}}$; and the family of Ore extensions studied in \cite{ArtamonovLezamaFajardo2016}.\\

Following Rosenberg \cite{Rosenberg1995}, Definition C4.3, a {\em 3-dimensional skew polynomial algebra $\cA$} is a
$\Bbbk$-algebra generated by the variables $x,y,z$ restricted to relations $
yz-\alpha zy=\lambda,\ zx-\beta xz=\mu$, and $xy-\gamma
yx=\nu$, such that the following conditions hold:
\begin{enumerate}
\item [\rm (1)] $\lambda, \mu, \nu\in \Bbbk+\Bbbk x+\Bbbk y+\Bbbk z$, and $\alpha, \beta, \gamma \in \Bbbk^{*}$;
\item [\rm (2)] Standard monomials $\{x^iy^jz^l\mid i,j,l\ge 0\}$ are a $\Bbbk$-basis of the algebra.
\end{enumerate}
3-dimensional skew polynomial algebras are very important in noncommutative algebraic geometry. Now, from the definition it is clear that these algebras are skew PBW extensions (as a matter of fact, in \cite{ReyesSuarez2017FEJM} the authors  proved algorithmically that 3-dimensional skew polynomial algebras are examples of skew PBW extensions).

There exists a classification of 3-dimensional skew polynomial algebras, see \cite{Rosenberg1995}, Theorem C.4.3.1. More precisely, if $\cA$ is a 3-dimensional skew polynomial algebra, then $\cA$
is one of the following algebras:
\begin{enumerate}
\item [\rm (a)] if $|\{\alpha, \beta, \gamma\}|=3$, then $\cA$ is defined by the relations $yz-\alpha zy=0,\ zx-\beta xz=0,\ xy-\gamma yx=0$.
\item [\rm (b)] if $|\{\alpha, \beta, \gamma\}|=2$ and $\beta\neq \alpha =\gamma =1$, then $\cA$ is one of the following algebras:
\begin{enumerate}
\item [\rm (i)] $yz-zy=z,\ \ \ zx-\beta xz=y,\ \ \ xy-yx=x${\rm ;}
\item [\rm (ii)] $yz-zy=z,\ \ \ zx-\beta xz=b,\ \ \ xy-yx=x${\rm ;}
\item [\rm (iii)] $yz-zy=0,\ \ \ zx-\beta xz=y,\ \ \ xy-yx=0${\rm ;}
\item [\rm (iv)] $yz-zy=0,\ \ \ zx-\beta xz=b,\ \ \ xy-yx=0${\rm ;}
\item [\rm (v)] $yz-zy=az,\ \ \ zx-\beta xz=0,\ \ \ xy-yx=x${\rm ;}
\item [\rm (vi)] $yz-zy=z,\ \ \ zx-\beta xz=0,\ \ \ xy-yx=0$,
\end{enumerate}
where $a, b$ are any elements of $\Bbbk$. All nonzero values of $b$
give isomorphic algebras.
\item [\rm (c)] If $|\{\alpha, \beta, \gamma\}|=2$ and $\beta\neq \alpha=\gamma\neq 1$, then $\cA$ is one of the following algebras:
\begin{enumerate}
\item [\rm (i)] $yz-\alpha zy=0,\ \ \ zx-\beta xz=y+b,\ \ \ xy-\alpha yx=0${\rm ;}
\item [\rm (ii)] $yz-\alpha zy=0,\ \ \ zx-\beta xz=b,\ \ \ xy-\alpha yx=0$.
\end{enumerate}
In this case, $b$ is an arbitrary element of $\Bbbk$. Again, any
nonzero values of $b$ give isomorphic algebras.
\item [\rm (d)] If $\alpha=\beta=\gamma\neq 1$, then $\cA$ is the algebra defined by the relations  $yz-\alpha zy=a_1x+b_1,\ zx-\alpha xz=a_2y+b_2,\ xy-\alpha yx=a_3z+b_3$. If $a_i=0\ (i=1,2,3)$, then all nonzero values of $b_i$ give isomorphic
algebras.
\item [\rm (e)] If $\alpha=\beta=\gamma=1$, then $\cA$ is isomorphic to one of the following algebras:
\begin{enumerate}
\item [\rm (i)] $yz-zy=x,\ \ \ zx-xz=y,\ \ \ xy-yx=z${\rm ;}
\item [\rm (ii)] $yz-zy=0,\ \ \ zx-xz=0,\ \ \ xy-yx=z${\rm ;}
\item [\rm (iii)] $yz-zy=0,\ \ \ zx-xz=0,\ \ \ xy-yx=b${\rm ;}
\item [\rm (iv)] $yz-zy=-y,\ \ \ zx-xz=x+y,\ \ \ xy-yx=0${\rm ;}
\item [\rm (v)] $yz-zy=az,\ \ \ zx-xz=z,\ \ \ xy-yx=0${\rm ;}
\end{enumerate}
Parameters $a,b\in \Bbbk$ are arbitrary,  and all nonzero values of
$b$ generate isomorphic algebras.
\end{enumerate}

\vspace{0.5cm}

\noindent {\bf \Large{Acknowledgements}}

\vspace{0.5cm}

The first author was supported by the research fund of Facultad de Ciencias, Universidad Nacional de Colombia, Bogot\'a, Colombia, HERMES CODE 41535.



\begin{thebibliography}{60}

\bibitem{AcostaLezama2015}Acosta J. P, Lezama O. Universal property of skew PBW extensions. Alg Dis Mthm 2015; 20(1): 1-12.

\bibitem{AndersonCamillo1998}Anderson D D, Camillo V. Armendariz rings and Gaussian rings. Commun Algebra 1998; 26(7): 2265-2272.

\bibitem{Annin2002}Annin S. Associated primes over skew polynomial rings. Commun Algebra 2002; 30(5): 2511-2528.

\bibitem{Artamonov2015}Artamonov V. A. Derivations of skew PBW extensions.  Commun Math Stat 2015; 3(4): 449-457.

\bibitem{ArtamonovLezamaFajardo2016}Artamonov V. A, Lezama O, Fajardo W.  Extended modules and Ore extensions.  Commun Math Stat 2016; 4(2): 189-202.

\bibitem{Bell1970}Bell H. E. Near-rings in which each element is a power of itself. B Aust Math Soc 1970; 2(3): 363-368.

\bibitem{Birkenmeieretal2013}Birkenmeier G. F, Park J. K, Rizvi  S. T. Extensions of Rings and Modules. New York: Springer-Verlag, 2013.

\bibitem{BrownGoodearl2002}Brown K. A,  Goodearl K. R. Lectures on Algebraic Quantum Groups. Barcelona: Birkha\"user, 2002.

\bibitem{BuesoTorrecillasVerschoren}Bueso J, G\'omez-Torrecillas J,  Verschoren A. Algorithmic Methods in Non-commutative Algebra: Applications to Quantum Groups. Dordrecht: Kluwer, 2003.

\bibitem{LezamaGallego2011}Gallego C, Lezama O. Gr\"obner bases for ideals of $\sigma$-PBW extensions. Commun Algebra 2011; 39(1): 50-75.

\bibitem{LezamaGallego2017}Gallego C, Lezama O. Projective modules and Gr\"obner bases for skew PBW extensions. Diss Math 2017; 521: 1-50.

\bibitem{HongKimKwak2000}Hong C. Y, Kim N. K, Kwak T. K. Ore extensions of Baer and p.p.-rings. J Pure Appl Algebra 2000; 151(3): 215-226.

\bibitem{HuhLeeSmoktunowicz2002}Huh C, Lee C, Smoktunowicz A. Armendariz rings and semicommutative rings.  Commun Algebra 2002; 30(2): 751-761.

\bibitem{JaramilloReyes2018}Jaramillo J, Reyes A. Symmetry and Reversibility Properties for Quantum Algebras and Skew Poincar\'e-Birkhoff-Witt extensions. Ingenier\'ia y Ciencia 2017; 14(27): 29-52.

\bibitem{KimLee2003}Kim N, Lee Y. Extensions of reversible rings. J Pure Appl Algebra 2003; 185(1-3): 207-223.

\bibitem{Krempa1996}Krempa J. Some examples of reduced rings. Algebr Colloq 1996; 3(4): 289-300.

\bibitem{Lambek1971}Lambek J. On the representation of modules by sheaves of factor modules. Can Math Bull 1971;  14: 359-368.

\bibitem{LezamaAcostaReyes2015}Lezama O, Acosta J. P, Reyes A. Prime ideals of skew PBW extensions. Rev Union Mat Argent 2015; 56(2): 39-55.

\bibitem{LezamaLatorre2017}Lezama O, Latorre E. Non-commutative algebraic geometry of semi-graded rings. Int J Algebr Comput 2017; 27(4): 361-389.

\bibitem{LezamaReyes2014}Lezama O, Reyes A. Some Homological Properties of Skew PBW Extensions. Commun Algebra 2014; 42(3): 1200-1230.

\bibitem{LiuZhao2006}Liu Z, Zhao R. On weak Armendariz rings. Commun Algebra 2006; 37(7): 2607-2616.

\bibitem{Marks2003}Marks G. A taxonomy of 2-primal rings. J Algebra 2003;  266(2): 494-520.

\bibitem{MoussaviHashemi2005}Moussavi A, Hashemi E. On $(\alpha,\delta)$-skew Armendariz rings. J Korean Math Soc. 2005; 42(2): 353-363.

\bibitem{NinoReyes2017}Ni\~no A, Reyes A. Some ring theoretical properties of skew Poincar\'e-Birkhoff-Witt extensions. Bol Mat 2017; 24(2): 131-148

\bibitem{OuyangChen2010}Ouyang, L,  Chen H. On weak symmetric rings. Commun Algebra 2010; 38(2): 697-713.

\bibitem{RegeChhawchharia1997}Rege M. B, Chhawchharia S. Armendariz rings. P Jpn Acad A-Math 1997; 73(1): 14-17.

\bibitem{ReyesPhD2013}Reyes A. Ring and Module Theoretical Properties of Skew PBW Extensions. PhD, National University of Colombia, Bogot\'a, Colombia, 2013.

\bibitem{Reyes2015}Reyes A. Skew PBW extensions of Baer, quasi-Baer, p.p. and p.q.-rings. Rev Integr Temas Mat 2015; 33(2): 173-189.

\bibitem{Reyes2018}Reyes A. $\sigma$-PBW extensions of skew $\Pi$-Armendariz rings. Far East J Math Sci 2018; 103(2): 401-428.

\bibitem{ReyesSuarezClifford2017}Reyes A, Su\'arez H. $\sigma$-PBW Extensions of Skew Armendariz Rings. Adv Appl Clifford Al 2017; 27(4), 3197-3224.

\bibitem{ReyesSuarez2017FEJM}Reyes A, Su\'arez H. PBW bases for some 3-dimensional skew polynomial algebras. Far East J. Math. Sci 2017; 101(6):  1207-1228.

\bibitem{ReyesSuarezUMA2018}Reyes A,  Su\'arez H. A notion of compatibility for Armendariz and Baer properties over skew PBW extensions. Rev Union Mat Argent 2018; 59(1): 157-178.

\bibitem{ReyesYesica2018}Reyes A,  Su\'arez Y. On the ACCP in skew Poincar\'e-Birkhoff-Witt extensions. Beitr Algebra Geom 2018, DOI 10.1007/s13366-018-0384-8

\bibitem{Rosenberg1995}Rosenberg A. Non-commutative Algebraic Geometry and Representations of Quantized Algebras. Mathematics and its applications.  Kluwer Academic Publishers, 1995.

\bibitem{Suarez2017}Su\'arez H.  Koszulity for graded skew PBW extensions. Commun Algebra 2017; 45(10): 4569-4580.

\bibitem{SuarezLezamaReyes2017}Su\'arez H, Lezama O, Reyes A. Calabi-Yau property for graded skew PBW extensions. Rev Colomb Math 2017; 51(2): 221-239.

\bibitem{SuarezReyesgenerKoszul2017}Su\'arez H, Reyes A. Koszulity for skew PBW extensions over fields. JP J Algebra Number Theory Appl 2017; 39(2): 181-203.

\end{thebibliography}
\end{document}